\documentclass[12pt, a4paper]{amsart}
\usepackage{fullpage}
\usepackage{amscd,amsmath,amsthm,amssymb,graphics,color}
\renewcommand{\Bbb}{\mathbb}
 \usepackage{thmtools} 
 \usepackage[colorlinks=true, pdfstartview=FitV, linkcolor=blue, citecolor=blue, urlcolor=blue]{hyperref}
\usepackage[capitalise, noabbrev]{cleveref}
\usepackage{tikz}
\usepackage{pdfpages}
\usepackage{paralist}
\usepackage{tkz-graph}
\usepackage{tikz-cd}
\usepackage{caption,subcaption}

\usetikzlibrary{backgrounds}
\pgfdeclarelayer{cubes}
\pgfsetlayers{cubes,main}
\colorlet{cubecolor}{gray}

\newcommand\cubefaces[4][cubecolor]{%
  \begin{pgfonlayer}{cubes}
    \fill[#1!18] (#2+1,#3,#4)--(#2+1,#3+1,#4)--(#2+1,#3+1,#4-1)--(#2+1,#3,#4-1)--cycle;
    \fill[#1!6]  (#2,#3+1,#4)--(#2+1,#3+1,#4)--(#2+1,#3+1,#4-1)--(#2,#3+1,#4-1)--cycle;
    \fill[#1!12] (#2,#3,#4)--(#2+1,#3,#4)--(#2+1,#3+1,#4)--(#2,#3+1,#4)--cycle;
  \end{pgfonlayer}
}

\newcommand\shadecube[3]{%
  \begin{pgfonlayer}{cubes}
    \fill[cubecolor!18] (#1+1,#2,#3)--(#1+1,#2+1,#3)--(#1+1,#2+1,#3-1)--(#1+1,#2,#3-1)--cycle;
    \fill[cubecolor!6]  (#1,#2+1,#3)--(#1+1,#2+1,#3)--(#1+1,#2+1,#3-1)--(#1,#2+1,#3-1)--cycle;
    \fill[cubecolor!12] (#1,#2,#3)--(#1+1,#2,#3)--(#1+1,#2+1,#3)--(#1,#2+1,#3)--cycle;
  \end{pgfonlayer}
}

\newcommand\unitcube[3]{%
  \cubefaces{#1}{#2}{#3}
  \draw (#1,#2,#3)--(#1+1,#2,#3)--(#1+1,#2,#3-1)--(#1,#2,#3-1)--cycle;
  \draw (#1,#2+1,#3)--(#1+1,#2+1,#3)--(#1+1,#2+1,#3-1)--(#1,#2+1,#3-1)--cycle;
  \draw (#1,#2,#3)--(#1,#2+1,#3);
  \draw (#1+1,#2,#3)--(#1+1,#2+1,#3);
  \draw (#1+1,#2,#3-1)--(#1+1,#2+1,#3-1);
  \draw (#1,#2,#3-1)--(#1,#2+1,#3-1);
}

\newcommand\unitcubebig[3]{%
  \cubefaces{#1}{#2}{#3}
  \draw (#1,#2,#3)--(#1+.33,#2,#3) (#1+.66,#2,#3)--(#1+1,#2,#3);
  \draw (#1,#2,#3-1)--(#1+.33,#2,#3-1) (#1+.66,#2,#3-1)--(#1+1,#2,#3-1);
  \draw (#1,#2+1,#3)--(#1+.33,#2+1,#3) (#1+.66,#2+1,#3)--(#1+1,#2+1,#3);
  \draw (#1,#2+1,#3-1)--(#1+.33,#2+1,#3-1) (#1+.66,#2+1,#3-1)--(#1+1,#2+1,#3-1);
  \draw (#1,#2,#3)--(#1,#2+.33,#3) (#1,#2+.66,#3)--(#1,#2+1,#3);
  \draw (#1,#2,#3-1)--(#1,#2+.33,#3-1) (#1,#2+.66,#3-1)--(#1,#2+1,#3-1);
  \draw (#1+1,#2,#3)--(#1+1,#2+.33,#3) (#1+1,#2+.66,#3)--(#1+1,#2+1,#3);
  \draw (#1+1,#2,#3-1)--(#1+1,#2+.33,#3-1) (#1+1,#2+.66,#3-1)--(#1+1,#2+1,#3-1);
  \draw (#1,#2+1,#3)--(#1,#2+1,#3-1);
  \draw (#1+1,#2+1,#3)--(#1+1,#2+1,#3-1);
  \draw (#1,#2,#3)--(#1,#2,#3-1);
  \draw (#1+1,#2,#3)--(#1+1,#2,#3-1);
  \draw[dotted,thick] (#1+.33,#2,#3)--(#1+.66,#2,#3);
  \draw[dotted,thick] (#1+.33,#2,#3-1)--(#1+.66,#2,#3-1);
  \draw[dotted,thick] (#1+.33,#2+1,#3)--(#1+.66,#2+1,#3);
  \draw[dotted,thick] (#1+.33,#2+1,#3-1)--(#1+.66,#2+1,#3-1);
  \draw[dotted,thick] (#1,#2+.33,#3)--(#1,#2+.66,#3);
  \draw[dotted,thick] (#1,#2+.33,#3-1)--(#1,#2+.66,#3-1);
  \draw[dotted,thick] (#1+1,#2+.33,#3)--(#1+1,#2+.66,#3);
  \draw[dotted,thick] (#1+1,#2+.33,#3-1)--(#1+1,#2+.66,#3-1);
}

\newcommand\unitcubeBig[3]{%
  \cubefaces{#1}{#2}{#3}
  \draw (#1,#2,#3)--(#1+.33,#2,#3) (#1+.66,#2,#3)--(#1+1,#2,#3);
  \draw (#1,#2,#3-1)--(#1+.33,#2,#3-1) (#1+.66,#2,#3-1)--(#1+1,#2,#3-1);
  \draw (#1,#2+1,#3)--(#1+.33,#2+1,#3) (#1+.66,#2+1,#3)--(#1+1,#2+1,#3);
  \draw (#1,#2+1,#3-1)--(#1+.33,#2+1,#3-1) (#1+.66,#2+1,#3-1)--(#1+1,#2+1,#3-1);
  \draw (#1,#2,#3)--(#1,#2+.33,#3) (#1,#2+.66,#3)--(#1,#2+1,#3);
  \draw (#1,#2,#3-1)--(#1,#2+.33,#3-1) (#1,#2+.66,#3-1)--(#1,#2+1,#3-1);
  \draw (#1+1,#2,#3)--(#1+1,#2+.33,#3) (#1+1,#2+.66,#3)--(#1+1,#2+1,#3);
  \draw (#1+1,#2,#3-1)--(#1+1,#2+.33,#3-1) (#1+1,#2+.66,#3-1)--(#1+1,#2+1,#3-1);
  \draw (#1,#2+1,#3)--(#1,#2+1,#3-.33) (#1,#2+1,#3-.66)--(#1,#2+1,#3-1);
  \draw (#1+1,#2+1,#3)--(#1+1,#2+1,#3-.33) (#1+1,#2+1,#3-.66)--(#1+1,#2+1,#3-1);
  \draw (#1,#2,#3)--(#1,#2,#3-.33) (#1,#2,#3-.66)--(#1,#2,#3-1);
  \draw (#1+1,#2,#3)--(#1+1,#2,#3-.33) (#1+1,#2,#3-.66)--(#1+1,#2,#3-1);
  \draw[dotted,thick] (#1+.33,#2,#3)--(#1+.66,#2,#3);
  \draw[dotted,thick] (#1+.33,#2,#3-1)--(#1+.66,#2,#3-1);
  \draw[dotted,thick] (#1+.33,#2+1,#3)--(#1+.66,#2+1,#3);
  \draw[dotted,thick] (#1+.33,#2+1,#3-1)--(#1+.66,#2+1,#3-1);
  \draw[dotted,thick] (#1,#2+.33,#3)--(#1,#2+.66,#3);
  \draw[dotted,thick] (#1,#2+.33,#3-1)--(#1,#2+.66,#3-1);
  \draw[dotted,thick] (#1+1,#2+.33,#3)--(#1+1,#2+.66,#3);
  \draw[dotted,thick] (#1+1,#2+.33,#3-1)--(#1+1,#2+.66,#3-1);
  \draw[dotted,thick] (#1,#2+1,#3-.33)--(#1,#2+1,#3-.66);
  \draw[dotted,thick] (#1+1,#2+1,#3-.33)--(#1+1,#2+1,#3-.66);
  \draw[dotted,thick] (#1,#2,#3-.33)--(#1,#2,#3-.66);
  \draw[dotted,thick] (#1+1,#2,#3-.33)--(#1+1,#2,#3-.66);
}

\usetikzlibrary{decorations.pathreplacing,calc}

\def\NZQ{\Bbb}               
\def\NN{{\NZQ N}}

\def\RR{{\NZQ R}}

\def\frk{\frak}               

\def\Phi{{\frk n}}
\def\Phi{{\frk N}}

\def\ML{{\mathcal L}}
\def\MP{{\mathcal P}}
\def\MQ{{\mathcal Q}}

\def\MS{{\mathcal S}}
\def\MG{{\mathcal G}}

\def\opn#1#2{\def#1{\operatorname{#2}}} 
\opn\chara{char} \opn\length{\ell} \opn\pd{pd} \opn\rk{rk}
\opn\projdim{proj\,dim} \opn\injdim{inj\,dim} \opn\rank{rank}
\opn\depth{depth} \opn\grade{grade} \opn\height{height}\opn\coheight{coheight}
\opn\embdim{emb\,dim} \opn\codim{codim}

\opn\Tr{Tr} \opn\bigrank{big\,rank}
\opn\superheight{superheight}\opn\lcm{lcm}
\opn\trdeg{tr\,deg}
\opn\reg{reg} \opn\lreg{lreg} \opn\ini{in} \opn\lpd{lpd}
\opn\size{size}\opn\bigsize{bigsize}
\opn\cosize{cosize}\opn\bigcosize{bigcosize}
\opn\sdepth{sdepth}\opn\sreg{sreg}
\opn\lk{lk}
\opn\del{del}
\opn\fdepth{fdepth}\opn\type{type}
\opn\GL{GL} 

\opn\div{div} \opn\Div{Div} \opn\cl{cl} \opn\Cl{Cl}
\opn\Spec{Spec} \opn\Supp{Supp} \opn\supp{supp} \opn\Sing{Sing}
\opn\Ass{Ass} \opn\Min{Min}\opn\Mon{Mon} \opn\dstab{dstab} \opn\astab{astab}
\opn\Syz{Syz}
\opn\Ann{Ann} \opn\Rad{Rad} \opn\Soc{Soc}
\opn\Im{Im} \opn\ker{ker} \opn\Coker{Coker} \opn\Am{Am}
\opn\Hom{Hom} \opn\Tor{Tor} \opn\Ext{Ext} \opn\End{End}
\opn\Aut{Aut} \opn\id{id}

\opn\nat{nat}
\opn\pff{pf}
\opn\Pf{Pf} \opn\GL{GL} \opn\SL{SL} \opn\mod{mod} \opn\ord{ord}
\opn\Gin{Gin} \opn\Hilb{Hilb}\opn\sort{sort}
\opn\Proj{Proj}
\opn\aff{aff} \opn\con{conv} \opn\relint{relint} \opn\st{st}
\opn\lk{lk} \opn\cn{cn} \opn\core{core} \opn\vol{vol}
\opn\link{link} \opn\star{star}\opn\lex{lex}
\opn\gr{gr}

\newcommand{\qand}{\quad \mbox{and} \quad}

\newcommand{\qforall}{\quad \mbox{for all} \quad}

\newcommand{\qwhere}{\quad \mbox{where} \quad}

\opn\dirlim{\underrightarrow{\lim}}
\opn\inivlim{\underleftarrow{\lim}}
{}

\let\dirsum=\oplus

\let\to=\rightarrow

\def\Implies{\ifmmode\Longrightarrow \else
	\unskip${}\Longrightarrow{}$\ignorespaces\fi}
\def\implies{\ifmmode\Rightarrow \else
	\unskip${}\Rightarrow{}$\ignorespaces\fi}
\def\iff{\ifmmode\Longleftrightarrow \else
	\unskip${}\Longleftrightarrow{}$\ignorespaces\fi}

\let\:=\colon
\theoremstyle{plain}
\newtheorem{Theorem}{Theorem}[section]
\newtheorem{Lemma}[Theorem]{Lemma}
\newtheorem{lemma}[Theorem]{Lemma}
\newtheorem{Setup}[Theorem]{Setup}
\newtheorem{Corollary}[Theorem]{Corollary}
\newtheorem{Proposition}[Theorem]{Proposition}
\newtheorem{Remark}[Theorem]{Remark}

\newtheorem{Example}[Theorem]{Example}

\newtheorem{Definition}[Theorem]{Definition}

\newtheorem{Notation}[Theorem]{Notation}

\title{Polycubes and polycube ideals}

\author{Ayesha Asloob Qureshi}
\address{Sabanci University, Faculty of Engineering and Natural Sciences, Orta Mahalle, Tuzla 34956, Istanbul, Turkey}
\email{aqureshi@sabanciuniv.edu, ayesha.asloob@sabanciuniv.edu}

\author {Dharm Veer}
\address{Department of Mathematics, Purdue University, 150 N. University Street, West Lafayette, IN 47907-2067, USA}
\email{dveer@purdue.edu}


\begin{document}

\begin{abstract} 
 To a polycube $\MP$ in $\mathbb R^n$, we associate an ideal $I_\MP$ generated by the join--meet relations of vertices that are corners of a cuboid of $\MP$. For $n=2$ these are the polyomino ideals. In contrast with the planar case, we show that for $n\ge3$ the ideal $I_\MP$ is prime if and only if $\MP$ is a cuboid. For convex polycubes in $\mathbb R^3$ of uniform thickness whose vertex set is a sublattice of $\mathbb N^3$, the defining relations form a quadratic Gr\"obner basis, and we characterize the Cohen--Macaulay ones in terms of a poset on the vertices of $\MP$; in this case the order complex of the poset is vertex decomposable. Finally, as an application, we show  that the polycube algebra of a thickened
$L$-convex polyomino is Koszul, and we characterize when it is Cohen--Macaulay.
\end{abstract}

	\maketitle	

\section{Introduction}

Polytopal, combinatorial, and poset-theoretic structures have played an important role in the study of toric and binomial ideals. A fundamental example is provided by defining ideals of Hibi rings. Let $L$ be a
finite distributive lattice and $S=K[x_\alpha:\alpha\in L]$ be a polynomial ring over a field $K$. The defining ideal of the Hibi ring associated to $L$ is generated
by the join--meet relations
\[
x_\alpha x_\beta
-
x_{\alpha\wedge\beta}x_{\alpha\vee\beta},
\qquad
\alpha,\beta\in L \text{ incomparable}.
\]
Hibi \cite{HIBI87} showed that the resulting algebra is a normal Cohen--Macaulay domain. Moreover, these quadratic relations form a Gr\"obner basis with respect to a suitable reverse lexicographic order (see \cite[Theorem~6.19]{HHO18}), and the corresponding initial ideal is the Stanley--Reisner ideal of the order complex of $L$.

In this paper, we ask what remains of this picture when the join--meet relations are imposed only \emph{locally}. A polycube $\MP$ is a finite face-connected collection of unit-cubes in $\mathbb R^n$, and its vertex set $V(\MP) = \MP \cap\NN^n$ inherits the coordinatewise order from $\NN^n$. Polycubes are natural higher-dimensional generalizations of polyominoes and have been studied extensively in combinatorics, particularly from the viewpoint of enumeration; see, for example, \cite{BBRpolycube, ABpolycube, MMR22} and the references therein.

An interval $I\subseteq\mathbb R^n$ is called a \emph{cuboid} of $\mathcal P$ if every unit-cube contained in $I$ belongs to $\mathcal P$. Given a polycube $\mathcal P$, we impose the join--meet relation for two incomparable vertices only when they belong to the corner lattice of a common cuboid contained in $\mathcal P$. The resulting ideal $I_{\mathcal P}$ is called the \emph{polycube ideal}, and $K[\mathcal P]=S/I_{\mathcal P}$ the corresponding \emph{polycube algebra}.

This construction simultaneously extends two familiar settings. For $n=2$, it recovers the polyomino ideals introduced in \cite{Qur12}, while if $\mathcal P$ itself is a cuboid, then $K[\mathcal P]$ is the Hibi ring associated to a product of chains. Thus polycube ideals interpolate between polyomino ideals and Hibi-type join--meet ideals. The behaviour between these two extremes, however, changes sharply with the dimension. In the plane, large classes of polyomino ideals are prime; in particular, the polyomino ideal of every simple polyomino is prime \cite{HM14,QSS17}. In contrast, in dimension at least three primeness is completely rigid.

\medskip
\noindent
\textbf{Theorem A (\cref{pro:primepolycube}).}
Let $\mathcal P$ be a polycube in $\mathbb R^n$, where $n\geq3$.
Then $I_{\mathcal P}$ is prime if and only if $\mathcal P$ is a
cuboid.
\medskip

Having established this sharp obstruction to primeness, we turn in dimension three to two properties that survive much more widely: quadratic Gr\"obner bases and Cohen--Macaulayness. We consider a natural class of convex polycubes whose vertex sets are sublattices of $\mathbb N^3$ and whose nonempty $y$-columns all have the same height $m$, see \Cref{setup:polycube}. For every polycube in this class, we prove that the local join--meet relations form a quadratic Gr\"obner basis with respect to the reverse lexicographic order used for Hibi rings (\cref{thm:GB}). In particular, $K[\mathcal P]$ is Koszul.

The corresponding initial ideal admits a useful combinatorial description. We construct a poset $Q_{\mathcal P}$ on $V(\mathcal P)$ by adding comparabilities between certain pairs of vertices which do not belong to the vertex lattice of a common cuboid. We prove that $\operatorname{in}(I_{\mathcal P}) = I_{\Delta(Q_{\mathcal P})}$, where $\Delta(Q_{\mathcal P})$ denotes the order complex of $Q_{\mathcal P}$. This allows us to translate Cohen--Macaulayness of the polycube algebra into a question about
the topology of an explicitly defined poset.

To describe the resulting criterion, let $\mathcal L_{\mathcal P}$ denote the set of vertices occurring in
cover relations of $Q_{\mathcal P}$ which are incomparable in the coordinatewise order on $\mathbb N^3$. Using
the theorem of Conca and Varbaro~\cite{CV20squarefree}, together with a detailed analysis of the order structure of $Q_{\mathcal P}$, we obtain the following characterization.

\medskip
\noindent
\textbf{Theorem B (\cref{thm:main}).}
Let $\mathcal P$ be a polycube satisfying \Cref{setup:polycube}. The following
conditions are equivalent:
\begin{enumerate}
    \item $K[\mathcal P]$ is Cohen--Macaulay for any field $K$;
    \item $\Delta(Q_{\mathcal P})$ is Cohen--Macaulay;
    \item $\Delta(Q_{\mathcal P})$ is vertex decomposable;
    \item $\mathcal P$ is a cuboid, or $m=1$ and
    $\mathcal L_{\mathcal P}$ is a chain in $Q_{\mathcal P}$.
\end{enumerate}
\medskip

Finally, we extend our results to $L$-convex polycubes. Using the construction of \cite{EHQR21}, we show that their polycube ideals
admit quadratic Gr\"obner bases and their polycube algebras are Koszul. We also obtain a complete characterization of when these
coordinate rings are Cohen--Macaulay (\cref{thm:LconvexCM}).

\section{Preliminaries}    

\subsection{Abstract simplicial complexes}
Let $K$ be a field and $\Delta$ be an abstract simplicial complex. Let $V(\Delta)$ denote the \emph{vertex set} of $\Delta$.
The elements of $\Delta$ are called \emph{faces}.
The dimension of a face $\sigma$, denoted by $\dim(\sigma)$, is $\# \sigma -1$. The dimension of $\Delta$ is defined as $\max\{ \dim(\sigma) \colon \sigma \in \Delta\}$.  The maximal faces of $\Delta$ with respect to inclusion are called \emph{facets}. $\Delta$ is called \emph{pure} if all its facets have the same dimension.
For a face $\sigma$ in $\Delta$, define the following subcomplexes of $\Delta$:
\begin{itemize}
    \item The \emph{link} of $\sigma$ in $\Delta$, denoted by 
    $\link_\Delta(\sigma)$, as $\{\tau \in \Delta: \sigma\cap \tau = \emptyset \qand \sigma \cup \tau \in \Delta\}$.
    \item The \emph{deletion} of $\sigma$ in $\Delta$, denoted by 
    $\del_\Delta(\sigma)$, as $\{\tau \in \Delta: \sigma\cap \tau = \emptyset \}$.
\end{itemize}
The \emph{join} of two simplicial complexes $\Delta_1$ and $\Delta_2$ on disjoint vertex sets, denoted by $\Delta_1 * \Delta_2$, is the simplicial complex whose faces are disjoint unions of faces of $\Delta_1$ and $\Delta_2$.   

Let $S = K[x_v \colon v\in V(\Delta)]$ be a polynomial ring. The \emph{Stanley--Reisner ideal} of $\Delta$ is the monomial ideal
\[
I_\Delta = \big( \prod_{v \in  \sigma} x_v \colon \sigma \subseteq V(\Delta) \qand \sigma \quad \text{is not a face of} \quad \Delta \big)\,.
\]
The quotient ring $K[\Delta] = S/I_\Delta$ is called the \emph{Stanley--Reisner ring} of $\Delta$ over $K$.

A simplicial complex $\Delta$ is said to be \emph{Cohen--Macaulay over a field $K$} if for every face $\sigma\in \Delta$ one has that $\Tilde{H}_i(\link_\Delta{\sigma}; K) = 0$ for all $i < \dim(\link_\Delta{\sigma})$. 
A simplicial complex is called \emph{Cohen--Macaulay} if it is Cohen--Macaulay over every field $K$.
Reisner~\cite{ReisnerCMrings} showed that this is equivalent to $K[\Delta]$ being Cohen--Macaulay. It is well known that if $\Delta$ is Cohen--Macaulay over $K$, then $\Delta$ is pure.

A vertex $x$ of a simplicial complex $\Delta$ is called a \emph{shedding vertex} if every facet of $\del_{\Delta}(x)$ is also a facet of $\Delta$. A pure simplicial complex $\Delta$ is said to be \emph{vertex decomposable} if one of the following conditions holds:
\begin{enumerate}
    \item $\Delta = \{\emptyset\}$, or $\Delta$ is a simplex; or
    \item there exists a shedding vertex $x\in V(\Delta)$ such that both $\link_\Delta(x)$ and $\del_\Delta(x)$ are vertex decomposable.
\end{enumerate} 
Vertex decomposable simplicial complexes were introduced by Provan and Billera~\cite{PB80vertexdecomposable}, who showed
that vertex decomposable simplicial complexes are shellable, and hence Cohen--Macaulay~\cite[Theorem 5.1.13]{BH93}. 

\subsection{Poset theory}
Throughout the paper, all posets will be finite.
Let $P$ be a poset with order relation $\leq$. 
Let $x,y\in P$.
We say the element $y$ \emph{covers} $x$, denoted by $x\lessdot y$, if $x< y$ and there is no element $z$ in $P$ such that $x <z<y$.
Elements $x$ and $y$ are said to be \emph{comparable} if $x\leq y$ or $y \leq x$; otherwise they are called \emph{incomparable}.
If $x\leq y$, then the \emph{open interval} $(x,y)_P$ is the subposet $\{z\in P: x < z < y\}$ of $P$. Similarly, one defines $(x,y]_P$, $[x,y)_P$, and the \emph{closed interval} $[x,y]_P$. 
The poset $P$ is \emph{bounded} if it contains a unique minimal element and a unique maximal element. 
The \emph{order dual} of $P$ is the poset $P^*$ on the same set with order relation $\leq_*$ defined by $x\leq_* y$ if and only if $y \leq x$.

A \emph{chain} of $P$ is a totally ordered subset of $P$. The \emph{length} of a chain $C$ is $\#C-1$. 
The \emph{rank} of $P$, denoted by $\rank(P)$, is the maximum length of chains in $P$. 
A poset is called \emph{pure} if all its maximal chains have the same length.  
A chain $x_1 < x_2 < \cdots < x_r$ is called \emph{saturated} if $x_i \lessdot x_{i+1}$ for every $i=1,\ldots,r-1$.
The {\em ordinal sum} $P\dirsum Q$ of two disjoint posets $P$ and $Q$ is the poset on the set $P\cup Q$ with the following order: for $x, y \in P\dirsum Q $, we have $x \leq y$ if either
\begin{itemize}
    \item $x, y \in P$ and $x \leq y$ in $P$, or
    \item $x, y \in Q$ and $x \leq y$ in $Q$, or
    \item $x \in P$ and $y \in Q$.
\end{itemize}  

The order complex $\Delta(P)$ of $P$ is a simplicial complex on the underlying set of $P$ whose faces are chains of $P$. The dimension of $\Delta(P)$ is the rank of $P$.
For the arguments in this article, a criterion for Cohen--Macaulayness in terms of the intervals of the poset will be useful.
Baclawski gave such a local, interval-based characterization
\cite{baclawskicmdefinition80}, later streamlined by Bj\"orner, Garsia, and Stanley
\cite{BGS82} into the following statement.

\begin{Lemma}\label{lem:intervalCM}
    Let $P$ be a bounded poset. Then, $\Delta(P)$ is Cohen--Macaulay over $K$ if and only if for all $x\leq y$ in $P$,
    \[
    \Tilde{H}_i(\Delta\big((x,y)_P \big); K) = 0 \qforall i < \rank\big((x,y)_P\big)\,.
    \]
\end{Lemma}

\subsection{Hibi ring} A poset $L$ is called a \textit{lattice} if, for every pair of elements $x,y\in L$, there exists a greatest lower bound, called the \textit{meet} and denoted by $x\wedge y$, and a least upper bound, called the \textit{join} and denoted by $x\vee y$. A lattice $L$ is called \emph{distributive} if for all $x,y,z\in L$, the following distributive laws hold:
\[
x\wedge(y\vee z)=(x\wedge y)\vee(x\wedge z) \qand
x\vee(y\wedge z)=(x\vee y)\wedge(x\vee z)\,.
\]

Let $L$ be a finite distributive lattice. 
Let $S = K[x_{\alpha} : \alpha \in L]$ be a polynomial ring. 
The {\em Hibi ring} associated to $L$ is $S/I_L$, where 
    \[
          I_L =  \big( x_{\alpha}x_{\beta} - x_{\alpha\wedge \beta}x_{\alpha\vee \beta} : \alpha,\beta \in L, \alpha\      \text{and}\ \beta \ \text{are incomparable}\big).
   \]
The ideal $I_L$ is called the {\em join-meet ideal} associated to $L$. A Hibi ring is a Koszul~\cite[Theorem \ 6.19]{HHO18} normal Cohen--Macaulay domain~\cite{HIBI87}.

\section{Polycubes and polycube ideals}\label{sec:polycubes}
In this section, we introduce polycubes and the binomial ideals associated to them, and study the primeness of these ideals. Set $\NN = \{0,1,2,\ldots\}$. Let $n\geq 2$ be an integer. For $a=(a_1,\ldots,a_n)\in\mathbb N^n$, we denote by $a_i$ the $i$-th coordinate of $a$. 

An $n$-dimensional unit-cube in
$\mathbb{R}^n$ is a set of the form
\[
\{(x_1,\ldots,x_n)\in\mathbb{R}^n:
a_r\leq x_r\leq a_r+1\text{ for all }1\leq r\leq n\},
\]
where $a=(a_1,\ldots,a_n)\in\mathbb{N}^n$. The vertex set of a unit-cube
$C$, denoted by $V(C)$, is $C\cap\mathbb{N}^n$. For a unit-cube
\[
C=[(a_1,\ldots,a_n),(a_1+1,\ldots,a_n+1)]
=[a,a+\mathbf{1}]\subseteq\mathbb{R}^n,
\]
the vertex $a=(a_1,\ldots,a_n)$ is called the \emph{lower-left vertex} of $C$. For $a\in\mathbb{N}^n$, we denote by $C(a)$ the unit-cube with lower-left vertex $a$. 

Let $\varepsilon\in\{a_s,a_s+1\}$, for some $1\leq s\leq n$.
An $(n-1)$-dimensional face (or face) of $C(a)$ is a set of the form
\[
\{(x_1,\ldots,x_n)\in\mathbb R^n:
x_s=\varepsilon,\;
a_r\leq x_r\leq a_r+1\text{ for all }r\neq s\}.
\]
Let $\MP$ be a finite collection of $n$-dimensional unit-cubes. The vertex set of $\MP$,
denoted by $V(\MP)$, is
\[
V(\MP)=\bigcup_{C\in \MP}V(C).
\]
Two unit-cubes $C,D\in \MP$ are called \emph{neighbours} if
$C\cap D$ is an $(n-1)$-dimensional face of both $C$ and $D$.
A path from $C$ to $D$ is a sequence of distinct unit-cubes
\[
C=C_0,C_1,\ldots,C_m=D
\]
such that $C_{i-1}$ and $C_i$ are neighbours for every
$1\leq i\leq m$. The non-empty collection $\MP$ of unit-cubes is called an
$n$-dimensional \emph{polycube} if there is a path between any two
unit-cubes of $\MP$. For $n=2$, a polycube is called a \emph{polyomino}.

We identify a finite collection $\MP$ of unit-cubes with the union of its
unit-cubes whenever $\MP$ is viewed as a subset of $\mathbb R^n$.
We say that $\MP$ is \emph{convex} if every line segment parallel to a
coordinate axis whose endpoints belong to $\MP$ is contained in $\MP$.

Let $a,b \in \mathbb{N}^n$ such that $a_r<b_r$ for every $1\leq r\leq n$.
The interval $[a,b]$ in $\RR^n$ is called a \emph{cuboid} (or an \emph{inner interval}) of $\MP$ if
\[
C(u)\in \MP \quad \text{for all}\quad u\in\mathbb{N}^n \text{ with } a\leq u\leq b-\bf{ 1 }.
\]
We say that $\MP$ is a \emph{cuboid} if there exists a cuboid
$C=[a,b]$ of $\MP$ such that
\[
V(\MP)=V(C)=[a,b]\cap\mathbb{N}^n.
\]
Equivalently, the unit-cubes of $\MP$ are precisely the unit-cubes contained in $C$.

We next associate a distributive lattice to each cuboid contained in
$\MP$, which will be used to define the polycube ideal. 
Let $C=[a,b]\subseteq\mathbb R^n$ be a cuboid contained in $P$. We define
\[
L_C=
\{(k_1,\ldots,k_n)\in\mathbb N^n:
k_r\in\{a_r,b_r\}\text{ for all }1\leq r\leq n\}.
\]
Then $L_C$ is a sublattice of $\mathbb N^n$, and hence a finite
distributive lattice. In fact, $L_C$ is isomorphic to the Boolean
lattice $B_n$, and it consists of the $2^n$ corners of $C$. Observe that if $C$ is a unit-cube, then
$L_C=V(C)$; otherwise $L_C \subsetneq V(C)$. The meet and join in $L_C$ are given by the coordinatewise minimum
and maximum, respectively. Thus, if $a,b\in L_C$, then
\[
a\wedge b=
(\min\{a_1,b_1\},\ldots,\min\{a_n,b_n\})
\qand
a\vee b=
(\max\{a_1,b_1\},\ldots,\max\{a_n,b_n\}).
\]
In particular, these operations agree with the meet and join in
$\mathbb N^n$ and therefore do not depend on the choice of a cuboid
containing $a$ and $b$.

\begin{figure}[h]
        \centering
\begin{tikzpicture}[scale=2.5]
\unitcubeBig{0}{0}{0};
    
\node [draw = blue, circle, fill=blue, inner sep=1pt, label=left:{\tiny{{$(a_1,a_2,a_3)$}}}] at (0,0,0) {};
\node [draw = blue, circle, fill=blue, inner sep=1pt, label=left:{\tiny{{$(a_1,b_2,a_3)$}}}] at (0,0,-1) {};
\node [draw = blue, circle, fill=blue, inner sep=1pt, label=left:{\tiny{{$(a_1,a_2,b_3)$}}}] at (0,1,0) {};
\node [draw = blue, circle, fill=blue, inner sep=1pt, label=left:{\tiny{{$(a_1,b_2,b_3)$}}}] at (0,1,-1) {};

\node [draw = blue, circle, fill=blue, inner sep=1pt, label=right:{\tiny{{$(b_1,b_2,a_3)$}}}] at (1,0,-1) {};
\node [draw = blue, circle, fill=blue, inner sep=1pt, label=right:{\tiny{{$(b_1,a_2,b_3)$}}}] at (1,1,0) {};  
\node [draw = blue, circle, fill=blue, inner sep=1pt, label=right:{\tiny{{$(b_1,b_2,b_3)$}}}] at (1,1,-1) {};
\node [draw = blue, circle, fill=blue, inner sep=1pt, label=right:{\tiny{{$(b_1,a_2,a_3)$}}}] at (1,0,0) {};
\end{tikzpicture}
        \caption{For a cuboid $C = [(a_1,a_2,a_3), (b_1,b_2,b_3)] \subset \RR^3$, the vertices belonging to $L_C$ are shown in blue.}\label{fig:cuboidinN3} 
\end{figure}
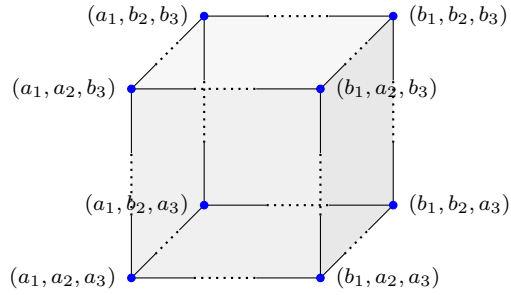

Let $S=K[x_a:a\in V(\MP)]$ be a polynomial ring over a field $K$. For a cuboid $C$ contained in
$\MP$, recall that $L_C$ is a finite distributive lattice. We associate
to $L_C$ the set
\[
\mathcal G_{L_C}
=
\{f_{a,b}=x_ax_b-x_{a\wedge b}x_{a\vee b}:
a,b\in L_C \text{ are incomparable}\}.
\]
By \cite[Theorem~6.19]{HHO18}, $\mathcal G_{L_C}$ is a Gröbner
basis of the defining ideal $J_{L_C}$ of the Hibi ring associated
to $L_C$ with respect to a suitable reverse lexicographic order.

For a finite collection $\MP$ of unit-cubes, set
\[
\mathcal G_\MP
=
\bigcup_{\substack{C\text{ a cuboid}\\ \text{contained in }\MP}}
\mathcal G_{L_C}, 
\]
and let $I_\MP= (\mathcal G_\MP)\subseteq S$. Thus, $I_\MP$ is generated by all binomials $f_{a,b}$ for which $a$ and $b$ are incomparable elements of $L_C$ for some cuboid $C$ contained in $\MP$. 
When $\MP$ is a polycube, $I_\MP$ is called the \emph{polycube ideal}
of $\MP$, and $K[\MP]=S/I_\MP$ is called the \emph{polycube algebra} associated to $\MP$. When $n=2$,
this construction agrees with the polyomino ideal and the associated
polyomino algebra introduced in \cite{Qur12}.

We first illustrate the construction in the simplest case of a single unit-cube.

\begin{Example}\label{ex:unitcube}
Let $\MP=\{C(0,0,0)\}\subseteq\mathbb R^3$, as shown in Figure~\ref{fig:example}. 

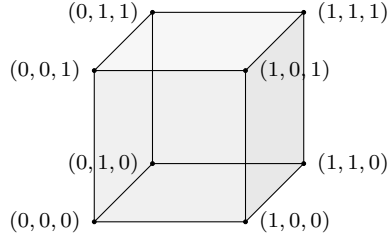
\begin{figure}[h]
        \centering
\begin{tikzpicture}[scale=2]
\unitcube{0}{0}{0};
    
\node [draw, circle, fill=black, inner sep=.5pt, label=left:{\tiny{{$(0,0,0)$}}}] at (0,0,0) {};
\node [draw, circle, fill=black, inner sep=.5pt, label=left:{\tiny{{$(0,1,0)$}}}] at (0,0,-1) {};
\node [draw, circle, fill=black, inner sep=.5pt, label=left:{\tiny{{$(0,0,1)$}}}] at (0,1,0) {};
\node [draw, circle, fill=black, inner sep=.5pt, label=left:{\tiny{{$(0,1,1)$}}}] at (0,1,-1) {};

\node [draw, circle, fill=black, inner sep=.5pt, label=right:{\tiny{{$(1,1,0)$}}}] at (1,0,-1) {};
\node [draw, circle, fill=black, inner sep=.5pt, label=right:{\tiny{{$(1,0,1)$}}}] at (1,1,0) {};  
\node [draw, circle, fill=black, inner sep=.5pt, label=right:{\tiny{{$(1,1,1)$}}}] at (1,1,-1) {};
\node [draw, circle, fill=black, inner sep=.5pt, label=right:{\tiny{{$(1,0,0)$}}}] at (1,0,0) {};
\end{tikzpicture}
        \caption{The unit-cube $C(0,0,0)$ and its vertices.}\label{fig:example} 
\end{figure}

Then $\MP$ is a polycube consisting of a single unit-cube with
\[
V(\MP)=
\{(0,0,0),(1,0,0),(0,1,0),(1,1,0),
(0,0,1),(1,0,1),(0,1,1),(1,1,1)\}.
\]
Since $L_{C(0,0,0)}=V(\MP)$, the ideal $I_\MP$
coincides with the defining ideal of the Hibi ring associated to
$V(\MP)$. For brevity, write $x_{ijk}=x_{(i,j,k)}$. The set $\mathcal G_\MP$ consists of the following nine
binomials:
\[
\begin{aligned}
&x_{001}x_{010}-x_{000}x_{011}, \qquad
x_{001}x_{100}-x_{000}x_{101}, \qquad
x_{010}x_{100}-x_{000}x_{110},\\
&x_{011}x_{110}-x_{010}x_{111}, \qquad
x_{101}x_{110}-x_{100}x_{111}, \qquad
x_{011}x_{101}-x_{001}x_{111},\\
&x_{010}x_{101}-x_{000}x_{111}, \qquad
x_{001}x_{110}-x_{000}x_{111}, \qquad
x_{011}x_{100}-x_{000}x_{111}.
\end{aligned}
\]
\end{Example}

The following lemma identifies a local configuration whose presence forces the polycube ideal to be non-prime. For convenience, let $\mathbf 0=(0,\ldots,0)\in\mathbb N^n$, and
for $1\leq i\leq n$, let $e_i$ denote the $i$th standard basis
vector of $\mathbb R^n$.

\begin{Lemma}\label{lem:not-prime}
Let $\MP$ be a polycube in $\mathbb R^n$, where $n\geq 3$, such that $C(\mathbf 0), C(e_1), C(e_1+e_3)$ belong to $\MP$, while $C(e_3)\notin \MP.$ Then $I_\MP$ is not prime.
\end{Lemma}

\begin{proof}
Set $p=e_3, \eta=\sum_{\substack{2\leq i\leq n\\i\neq3}}e_i,$
and
\[
u=p+\eta,\qquad
v=p+e_1,\qquad
q=p+e_1+\eta,\qquad
w=v+e_3,\qquad r=q+e_3\,.
\]
The vertices $u$ and $v$ are incomparable elements of
$L_{C(\mathbf 0)}$, with $u\wedge v=p, u\vee v=q.$ Hence $f=x_ux_v-x_px_q\in\mathcal G_\MP.$ Similarly, $q$ and $w$ are incomparable elements of
$L_{C(v)}$, with $q\wedge w=v, q\vee w=r.$ Thus, $g=x_qx_w-x_vx_r\in\mathcal G_\MP.$ Consequently,
\[
x_wf+x_pg
=x_v(x_ux_w-x_px_r)\in I_\MP.
\]

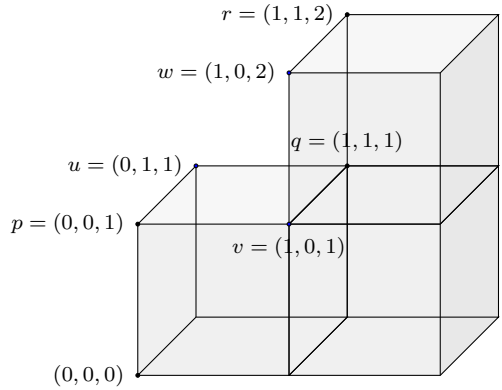
\begin{figure}[htbp]
\centering
\begin{tikzpicture}[scale=2]

\unitcube{0}{0}{0};
\unitcube{1}{0}{0};
\unitcube{1}{1}{0};

\node[draw,circle,fill=black,inner sep=.5pt,
      label=left:{\tiny $p=(0,0,1)$}] at (0,1,0) {};

\node[draw,circle,fill=blue,inner sep=.5pt,
      label=left:{\tiny $u=(0,1,1)$}] at (0,1,-1) {};

\node[draw,circle,fill=blue,inner sep=.5pt,
      label=below:{\tiny $v=(1,0,1)$}] at (1,1,0) {};

\node[draw,circle,fill=black,inner sep=.5pt,
      label=above:{\tiny $q=(1,1,1)$}] at (1,1,-1) {};

\node[draw,circle,fill=blue,inner sep=.5pt,
      label=left:{\tiny $w=(1,0,2)$}] at (1,2,0) {};

\node[draw,circle,fill=black,inner sep=.5pt,
      label=left:{\tiny $r=(1,1,2)$}] at (1,2,-1) {};

\node [draw, circle, fill=black, inner sep=.5pt, label=left:{\tiny{{$(0,0,0)$}}}] at (0,0,0) {};
\end{tikzpicture}
\caption{The configuration of Lemma~\ref{lem:not-prime} in dimension $3$. }
\end{figure}
Set $h=x_ux_w-x_px_r.$ The vertices $u$ and $w$ differ in every coordinate. Hence, if
$u,w\in L_D$ for some cuboid $D\subseteq P$, then necessarily
\[
D=[u\wedge w,u\vee w]=[p,r]=C(e_3),
\]
contrary to the assumption that $C(e_3)\notin P$. Therefore
$x_ux_w$ does not occur in any element of $\mathcal G_\MP$. Since
$I_\MP$ is generated by homogeneous quadrics, it follows that
$h\notin I_\MP$. Moreover, $x_v\notin I_\MP$, since $I_\MP$ is generated by homogeneous
quadrics. Thus $x_vh\in I_\MP, x_v\notin I_\MP, h\notin I_\MP,$ and therefore $I_\MP$ is not prime.
\end{proof}

Next we characterize the polycubes whose polycube
ideals are prime.

\begin{Proposition}\label{pro:primepolycube}
Let $\MP$ be a polycube in $\RR^n$, where $n\geq 3$. Then
$I_{\MP}$ is prime if and only if $\MP$ is a cuboid.
\end{Proposition}

\begin{proof}
Suppose that $P$ is a cuboid. Then $V(P)$ is a finite distributive
lattice. Moreover, every incomparable pair $a,b\in V(P)$ belongs to
$L_D$ for some cuboid $D\subseteq P$. Hence $I_P$ coincides with the
defining ideal of the Hibi ring associated to $V(P)$, and therefore
$I_P$ is prime by \cite{HIBI87}.

Conversely, suppose that $\MP$ is not a cuboid. We first claim that there exist $a\in\NN^n$, distinct
$i,j\in\{1,\ldots,n\}$, and $\varepsilon,\delta\in\{-1,1\}$ such
that
\[
C(a),\quad C(a+\varepsilon e_i),\quad
C(a+\varepsilon e_i+\delta e_j)\in\MP, \quad \text{ but }
\quad C(a+\delta e_j)\notin\MP,
\]
where all four lower-left vertices belong to $\NN^n$.

Indeed, suppose that no such configuration exists. Choose a cuboid
$Q=[p,q]\subseteq\MP$ which is maximal with respect to inclusion.
Since $\MP$ is not a cuboid and is connected, there are neighbouring
unit-cubes $C(d)\in Q$ and $C(d+\varepsilon e_j)\in\MP\setminus Q$.
Assume $\varepsilon=1$; the other case is analogous. Since $C(d+e_j)$ lies immediately outside $Q$ in the positive
$j$th coordinate direction, we have $d_j=q_j-1$. Consider the
boundary layer of unit-cubes of $Q$ given by
\[
\{C(u)\subseteq Q:u_j=q_j-1\}.
\]
Let $C(d')$ be any cube in this layer. Since $d_j=d'_j=q_j-1$,
we may pass from $d$ to $d'$ by changing the coordinates different
from $j$ one unit at a time. Hence there exist $d=d_0,d_1,\ldots,d_t=d'$ such that $C(d_k)\subseteq Q$ for every $k$ and
\[
d_{k+1}-d_k=\pm e_{i_k}
\qquad\text{for some }i_k\neq j.
\]

Suppose that $C(d_k+e_j)\in \MP$. Since
$C(d_k),C(d_{k+1})\in Q\subseteq \MP$, if
$C(d_{k+1}+e_j)\notin \MP$, then these three cubes together with the
missing cube $C(d_{k+1}+e_j)$ form the configuration excluded by our
assumption. Hence $C(d_{k+1}+e_j)\in \MP$. Inductively,
$C(d'+e_j)\in \MP$. Since $d'$ was arbitrary, every unit-cube in the layer immediately
adjacent to this side of $Q$ belongs to $\MP$. Therefore
$[p,q+e_j]$ is a cuboid contained in $\MP$, properly containing $Q$,
a contradiction.

By the claim, after a translation, a permutation of the coordinate
axes, and, if necessary, reversing the relevant coordinate directions,
we may assume that $\MP$ contains
\[
C(0,\ldots,0),\qquad
C(1,0,\ldots,0),\qquad
C(1,0,1,0,\ldots,0),
\]
but does not contain $C(0,0,1,0,\ldots,0).$ Hence $I_\MP$ is not prime by Lemma~\ref{lem:not-prime}.
\end{proof}

\section{Gröbner bases}\label{sec:gb}
For the remainder of the paper, we restrict our attention to polycubes in $\RR^3$. In this section, we study the Gröbner bases of polycube ideals associated with a particular class of polycubes described in the following setup. For fixed $i,k\in\mathbb{N}$, we call 
\[\{C(i,j,k)\in \MP:j\in\mathbb{N}\}\]
the $y$-column of $\MP$ at $(i,k)$.

\begin{Setup}\label{setup:polycube}
Fix an integer $m\geq 1$. Let $\mathcal \MP$ be a convex polycube in
$\mathbb{R}^3$ satisfying the following conditions:
\begin{enumerate}
    \item $V(\mathcal \MP)$ is a sublattice of $\mathbb{N}^3$;
    
    \item every nonempty $y$-column of $\MP$ is of the form\[    \{C(i,j,k):0\leq j\leq m-1\}\]for some $i,k\in\mathbb{N}$.
\end{enumerate}
\end{Setup}

Condition~(2) means that every nonempty $y$-column of $\MP$
    consists of exactly $m$ unit-cubes. Thus, $\MP$ may be viewed as a
    two-dimensional configuration in the $xz$-plane extended uniformly
    in the $y$-direction.

We will frequently use the projection onto the $xz$-plane. Let
\[
\pi:\mathbb R^3\longrightarrow\mathbb R^2,\qquad
\pi(x,y,z)=(x,z).
\]
We use the same notation for its restriction to $V(\MP)$. Thus, for
$a=(a_1,a_2,a_3)\in V(\MP)$, we have
$\pi(a)=(a_1,a_3)$. Moreover, if $D=[(i_1,j_1,k_1),(i_2,j_2,k_2)]$ is a cuboid, then its $xz$-projection is the rectangle $\pi(D)=[(i_1,k_1),(i_2,k_2)]\subseteq\mathbb R^2.$

We first record an immediate consequence of Setup~\ref{setup:polycube}.
\begin{Remark}\label{rem:structure}
Let $\MP$ satisfy Setup~\ref{setup:polycube}.  
\begin{enumerate} 
\item If $(i,j,k)\in V(\MP)$, then $(i,r,k)\in V(\MP)$ for every $0\leq r\leq m.$

\item If $D\subseteq \MP$ is a cuboid and $a,b \in V(\MP)$ with
$\pi(a),\pi(b)\in\pi(D)$, then $a,b\in L_E$ for some cuboid
$E\subseteq \MP$. Indeed, the full $y$-extension of $\pi(D)$ is
contained in $\MP$.
\end{enumerate}
\end{Remark}

The following observation shows that for a polycube as in \cref{setup:polycube}, incomparability is enough to
produce a relation in $\mathcal G_\MP$ whenever it is already visible
after projection onto the $xz$-plane.

\begin{Lemma}\label{lem:projection}
Let $\MP$ satisfy Setup~\ref{setup:polycube}, and let $a,b\in V(\MP)$
be incomparable. If $\pi(a)$ and $\pi(b)$ are incomparable in
$\mathbb N^2$, then $f_{a,b}\in\mathcal G_\MP$.
\end{Lemma}

\begin{proof}
Write $a=(i_1,j_1,k_1)$ and $b=(i_2,j_2,k_2)$. Since
$\pi(a)$ and $\pi(b)$ are incomparable, after interchanging $a$ and
$b$ if necessary, we may assume that $i_1<i_2$ and $k_1>k_2$.

By Remark~\ref{rem:structure},
$(i_1,0,k_1),(i_2,0,k_2)\in V(\MP)$. Since $V(\MP)$ is a sublattice of
$\mathbb N^3$, it follows that
$(i_1,0,k_2),(i_2,0,k_1)\in V(\MP)$. By convexity, first in the
$z$-direction and then in the $x$-direction, the rectangle $[(i_1,k_2),(i_2,k_1)]$ is contained in $\MP$. Since every nonempty $y$-column of $\MP$ consists
of exactly $m$ unit-cubes, the cuboid $[(i_1,0,k_2),(i_2,m,k_1)]$ is contained in $\MP$.

If $j_1\neq j_2$, then $a, b \in L_D$, where $D$ is the cuboid
\[
[(i_1,\min\{j_1,j_2\},k_2),
  (i_2,\max\{j_1,j_2\},k_1)].
\]
If $j_1=j_2$, choose $j'\in\{0,\ldots,m\}$ with $j'\neq j_1$ and use
the cuboid whose $y$-coordinates are $j_1$ and $j'$. In either case,
$a$ and $b$ belong to the lattice associated to a cuboid contained in
$\MP$. Hence $f_{a,b}\in\mathcal G_\MP$.
\end{proof}

We next identify the auxiliary relations that arise when reducing the
$S$-polynomial of two generators of $\mathcal G_\MP$ having a common
variable. Lemma~\ref{lem:projection} will allow us to verify that all
such relations still belong to $\mathcal G_\MP$.

\begin{Lemma}\label{lem:reduction-relations}
Let $\MP$ satisfy Setup~\ref{setup:polycube}, and suppose that
$f_{a,b},f_{b,c}\in\mathcal G_\MP$. Write
$a=(a_1,a_2,a_3)$ and $c=(c_1,c_2,c_3)$, and assume that
$a_2\geq c_2$. Set
\[
u=a\wedge b,\quad s=a\vee b,\quad
t=b\wedge c,\quad w=b\vee c, \quad 
v=c\wedge s,\quad z=a\vee t.
\]
Then, whenever the corresponding pair is incomparable, each of $f_{c,s}, f_{a,t}, f_{u,v}, f_{z,w}$ belongs to $\mathcal G_\MP$.
\end{Lemma}

\begin{proof}
Choose cuboids $D_1,D_2\subseteq \MP$ such that
$a,b\in L_{D_1}$ and $b,c\in L_{D_2}$, and set
\[
R_i=\pi(D_i),\qquad i=1,2.
\]
By Setup~\ref{setup:polycube}, if
$R_i=[(\alpha_i,\gamma_i),(\beta_i,\delta_i)]$, then the cuboid $[(\alpha_i,0,\gamma_i),(\beta_i,m,\delta_i)]$ is contained in $\MP$. Consequently, if two incomparable vertices have
their $xz$-projections in the same $R_i$, then they belong to $L_D$
for some cuboid $D\subseteq \MP$, and hence the corresponding binomial
belongs to $\mathcal G_\MP$.

If the projections of any of the pairs
$\{c,s\}$, $\{a,t\}$, $\{u,v\}$, or $\{z,w\}$ are incomparable in
$\mathbb N^2$, then the corresponding assertion follows from
Lemma~\ref{lem:projection}. It therefore remains to consider the case
when the projections of the corresponding pair are comparable.

Suppose first that $c$ and $s$ are incomparable, and write
$s=(s_1,s_2,s_3)$. Since $c_2\leq a_2\leq s_2$, we cannot have
$\pi(c)\leq\pi(s)$, for otherwise $c\leq s$. Hence
$\pi(s)<\pi(c)$. Since $b\leq s$, we obtain $\pi(b)\leq\pi(s)<\pi(c).$ Thus $\pi(s),\pi(c)\in R_2$, and consequently
$f_{c,s}\in\mathcal G_\MP$.

Next suppose that $a$ and $t$ are incomparable, and write
$t=(t_1,t_2,t_3)$. Since $t_2\leq c_2\leq a_2$, we cannot have
$\pi(t)\leq\pi(a)$, for otherwise $t\leq a$. Hence
$\pi(a)<\pi(t)$. Since $t\leq b$, we have $\pi(a)<\pi(t)\leq\pi(b).$ Thus $\pi(a),\pi(t)\in R_1$, and hence
$f_{a,t}\in\mathcal G_\MP$.

Now suppose that $u$ and $v$ are incomparable. Recall that
$u,s\in L_{D_1}$ and $t,b\in L_{D_2}$. If
$\pi(u)\leq\pi(v)$, then $v\leq s$, and therefore $\pi(u)\leq\pi(v)\leq\pi(s).$ Thus $\pi(u),\pi(v)\in R_1$. If instead
$\pi(v)\leq\pi(u)$, then $t\leq v$ and $u\leq b$, so $\pi(t)\leq\pi(v)\leq\pi(u)\leq\pi(b).$ Hence $\pi(u),\pi(v)\in R_2$. In either case, $f_{u,v}\in\mathcal G_\MP$.

Finally, suppose that $z$ and $w$ are incomparable. If
$\pi(z)\leq\pi(w)$, then $t\leq z$, and hence $\pi(t)\leq\pi(z)\leq\pi(w).$ Since $t,w\in L_{D_2}$, we have $\pi(z),\pi(w)\in R_2$. If instead
$\pi(w)\leq\pi(z)$, then $b\leq w$ and $z\leq s$, so $\pi(b)\leq\pi(w)\leq\pi(z)\leq\pi(s).$ Since $b,s\in L_{D_1}$, we have $\pi(z),\pi(w)\in R_1$. Thus $f_{z,w}\in\mathcal G_\MP$.
\end{proof}

We next fix the monomial order that will be used in the Gröbner basis
argument. Let $S=K[x_a:a\in V(\MP)]$. Choose a total order $\prec$ on the variables
of $S$ such that
\[
x_a\prec x_b\qquad\text{whenever }a<b\text{ in }V(\MP).
\]
Let $<_{\mathrm{rev}}$ denote the graded reverse lexicographic order
induced by $\prec$. Since $V(\MP)$ is a finite distributive lattice and $\prec$ extends its
partial order, for every $f_{a,b}\in\mathcal G_\MP$ we have
\[
\operatorname{in}_{<_{\mathrm{rev}}}(f_{a,b})=x_ax_b,
\]
see \cite[Theorem~6.19]{HHO18}.
We are now ready to prove the main result of this section.

\begin{Theorem}\label{thm:GB}
Let $\MP$ be a polycube satisfying Setup~\ref{setup:polycube}. Then
$\mathcal G_\MP$ is a Gröbner basis of $I_\MP$ with respect to
$<_{\mathrm{rev}}$.
\end{Theorem}

\begin{proof}
We use Buchberger's criterion. Let $f,g\in\mathcal G_\MP$ and $f\neq g$. If
$\operatorname{in}(f)$ and $\operatorname{in}(g)$ are relatively prime,
then $S(f,g)$ reduces to zero. Otherwise, after relabeling the vertices,
we may write
\[
f=f_{a,b}=x_ax_b-x_ux_s,\qquad
g=f_{b,c}=x_bx_c-x_tx_w,
\]
where $u=a\wedge b$, $s=a\vee b$, $t=b\wedge c$, and $w=b\vee c$.
After interchanging $f$ and $g$ if necessary, we may further assume
that $a_2\geq c_2$.

Set $v=c\wedge s$, $z=a\vee t$, $p=a\wedge b\wedge c$,
$q=a\vee b\vee c$, and $r=u\vee v$. By distributivity,
\[
u\wedge v=p,\qquad
u\vee v=r=z\wedge w,\qquad
z\vee w=q.
\]
Consider the binomials
\[
\begin{aligned}
h_1&=x_cx_s-x_vx_q, &
h_2&=x_ax_t-x_px_z,\\
h_3&=x_ux_v-x_px_r, &
h_4&=x_zx_w-x_rx_q.
\end{aligned}
\]
By Lemma~\ref{lem:reduction-relations}, whenever the corresponding
pair is incomparable, $h_i\in\mathcal G_\MP$. If the corresponding pair
is comparable, then $h_i=0$. Thus every nonzero $h_i$ may be used in
the following reductions. A direct computation gives
\[
S(f,g)=-x_cx_ux_s+x_ax_tx_w.
\]
Using $h_1$ and $h_2$, this reduces to
\[
R=-x_ux_vx_q+x_px_zx_w.
\]
Finally,
\[
R=-x_qh_3+x_ph_4,
\]
and hence $R$ reduces to zero using $h_3$ and $h_4$. Therefore
$S(f,g)$ reduces to zero. Thus every $S$-polynomial of two elements of $\mathcal G_\MP$ reduces
to zero, as required.
\end{proof}

\section{Cohen--Macaulayness and vertex decomposability}

Throughout this section, let ${\mathcal P}$ be a polycube satisfying
Setup~\ref{setup:polycube}. By Theorem~\ref{thm:GB},
$\mathcal G_{\mathcal P}$ is a quadratic Gröbner basis of $I_{\mathcal P}$ with respect to $<_{\mathrm{rev}}$, and hence
$\operatorname{in}_{<_{\mathrm{rev}}}(I_{\mathcal P})$ is a squarefree quadratic monomial ideal. It is known that
\[
S/I_{\mathcal P} \text{ is Cohen--Macaulay}
\quad\Longleftrightarrow\quad
S/\operatorname{in}_{<_{\mathrm{rev}}}(I_{\mathcal P})
\text{ is Cohen--Macaulay}.
\]
Indeed, the implication ``$\Leftarrow$'' holds for arbitrary Gr\"obner
degenerations~\cite[Theorem~2.19]{HHO18}, 
while ``$\Rightarrow$'' is the theorem of Conca and Varbaro on squarefree
Gr\"obner degenerations \cite{CV20squarefree}. 
Consequently, in order to characterize the Cohen--Macaulay polycube algebras in this
class, it suffices to determine when
$S/\operatorname{in}_{<_{\mathrm{rev}}}(I_{\mathcal P})$ is Cohen--Macaulay.

Since $\operatorname{in}_{<_{\mathrm{rev}}}(I_{\mathcal P})$ is generated by squarefree quadratic monomials, it is the Stanley--Reisner ideal of a flag complex on $V(\mathcal P)$ whose minimal nonfaces are the pairs $\{a,b\}$ such that $x_ax_b$ is the leading term of an element of $\MG_{\mathcal P}$. By \cref{thm:GB}, these are exactly the pairs $a,b$ which belong to
$L_D$ for some cuboid $D\subseteq\mathcal P$ and are incomparable. We wish to realize this flag complex as the order complex of a poset on $V(\mathcal P)$, so that the tools of poset topology become available.
  For a cuboid $D\subseteq {\mathcal P}$, let $\pi(D)$ denote its $xz$-projection, as in Section~\ref{sec:gb}.

\begin{Definition}\label{def:QP}
Define a relation $\leq_{Q_{\mathcal P}}$ on $V({\mathcal P})$ as follows. For
\[
a=(i_1,j_1,k_1),\qquad b=(i_2,j_2,k_2)\in V({\mathcal P}),
\]
we set $a\leq_{Q_{\mathcal P}}b$ if either
\begin{enumerate}
    \item $a\leq b$ in $\mathbb N^3$; or
    \item there is no cuboid $D\subseteq {\mathcal P}$ such that
    $a,b\in L_D$, and
    \[
    \pi(a)\leq\pi(b)\qquad\text{in} \quad \mathbb{N}^2.
    \]
\end{enumerate}
\end{Definition}

Observe that condition~(1) immediately implies that
\[
a\leq b \text{ in } V(\MP)
\quad\Longrightarrow\quad
a\leq_{Q_{\mathcal P}} b.
\]

It follows that $\leq_{Q_{\mathcal P}}$ is a partial order. Indeed, suppose that
$a\leq_{Q_{\mathcal P}}b$ and $b\leq_{Q_{\mathcal P}}a$. Then
$\pi(a)=\pi(b)$. By Remark~\ref{rem:structure}, $a$ and $b$ belong
to $L_D$ for some cuboid $D\subseteq {\mathcal P}$. Hence condition~(2) cannot
apply, so $a\leq b$ and $b\leq a$ in $\mathbb N^3$. Thus $a=b$.

For transitivity, suppose that
$a\leq_{Q_{\mathcal P}}b\leq_{Q_{\mathcal P}}c$. Then $\pi(a)\leq\pi(b)\leq\pi(c).$ If $a_2\leq c_2$, then $a\leq c$ in $\mathbb N^3$. Suppose that
$a_2>c_2$. If $a,c\in L_D$ for some cuboid $D\subseteq {\mathcal P}$, then
$\pi(b)\in\pi(D)$. By Remark~\ref{rem:structure}, there exist cuboids
$D_1,D_2\subseteq {\mathcal P}$ such that $a,b\in L_{D_1}$ and $b,c\in L_{D_2}.$ Since $a_2>c_2$, either $a_2>b_2$ or $b_2>c_2$.
In the first case, condition~(1) cannot give
$a\leq_{Q_{\mathcal P}}b$, while the existence of $D_1$ excludes
condition~(2), a contradiction. The second case similarly
contradicts $b\leq_{Q_{\mathcal P}}c$. Therefore $a$ and $c$ do not belong to $L_D$ for any cuboid
$D\subseteq {\mathcal P}$. Since $\pi(a)\leq\pi(c)$, condition~(2) gives
$a\leq_{Q_{\mathcal P}}c$. 

We denote the resulting poset by $Q_{\mathcal P}$. Let $\Delta(Q_{\mathcal P})$ be its order complex, and let $I_{\Delta(Q_{\mathcal P})}$ denote the Stanley--Reisner
ideal of $\Delta(Q_{\mathcal P})$.

\begin{Lemma}\label{lem:initial-order-complex}
We have $I_{\Delta(Q_{\mathcal P})}
=
\operatorname{in}_{<_{\mathrm{rev}}}(I_{\mathcal P}).$
\end{Lemma}

\begin{proof}
The Stanley--Reisner ideal $I_{\Delta(Q_{\mathcal P})}$ is generated by the
monomials $x_ax_b$, where $a$ and $b$ are incomparable in $Q_{\mathcal P}$. Let $f_{a,b}\in\mathcal G_{\mathcal P}$. Then $a$ and $b$ are incomparable in
$\mathbb N^3$ and belong to $L_D$ for some cuboid $D\subseteq {\mathcal P}$. Hence $a$ and $b$ are incomparable in $Q_{\mathcal P}$, and therefore $\ini_{<_{\mathrm{rev}}}(f_{a,b})=x_ax_b \in I_{\Delta(Q_{\mathcal P})}.$ Thus, Theorem~\ref{thm:GB} gives $\ini_{<_{\mathrm{rev}}}(I_{\mathcal P})
\subseteq I_{\Delta(Q_{\mathcal P})}.$

Conversely, let $a,b\in Q_{\mathcal P}$ be incomparable. Then $a$ and $b$ are
incomparable in $\mathbb N^3$, since otherwise condition~(1) would
make them comparable in $Q_{\mathcal P}$.

Suppose first that $\pi(a)$ and $\pi(b)$ are comparable in
$\mathbb N^2$. After interchanging $a$ and $b$ if necessary, assume
that $\pi(a)\leq\pi(b)$. If there were no cuboid $D\subseteq {\mathcal P}$ such
that $a,b\in L_D$, then condition~(2) of
Definition~\ref{def:QP} would give $a\leq_{Q_{\mathcal P}}b$, a contradiction.
Hence $a,b\in L_D$ for some cuboid $D\subseteq {\mathcal P}$, and therefore
$f_{a,b}\in\mathcal G_{\mathcal P}$.

If $\pi(a)$ and $\pi(b)$ are incomparable in $\mathbb N^2$, then
Lemma~\ref{lem:projection} gives $f_{a,b}\in\mathcal G_{\mathcal P}$.

Thus, in either case,
\[
x_ax_b=\operatorname{in}_{<_{\mathrm{rev}}}(f_{a,b})
\in\operatorname{in}_{<_{\mathrm{rev}}}(I_{\mathcal P}).
\]
Therefore
$I_{\Delta(Q_{\mathcal P})}=\operatorname{in}_{<_{\mathrm{rev}}}(I_{\mathcal P})$.
\end{proof}

\subsection{Necessary conditions for Cohen--Macaulayness of $\Delta(Q_\MP)$}

By \cref{lem:initial-order-complex} and Reisner's criterion, $S/{{\ini}_{<_{\rm rev}}(I_{\MP})}$ is Cohen--Macaulay if and only if $\Delta(Q_{\MP})$ is Cohen--Macaulay. Together with the equivalence recalled at the beginning of this section,
the same is true for $K[\MP]$. We now derive necessary conditions for $\Delta(Q_{\MP})$ to be Cohen--Macaulay.

\begin{Proposition}\label{prop:height-one} 
Let ${\mathcal P}$ satisfy Setup~\ref{setup:polycube}, and suppose that ${\mathcal P}$ is
not a cuboid. If $S/\operatorname{in}_{<_{\mathrm{rev}}}(I_{\mathcal P})$ is
Cohen--Macaulay, then $m=1$. In particular, after a translation in
the $x$- and $z$-coordinates, ${\mathcal P}\subseteq[(0,0,0),(l,1,n)]$ for some $l,n\geq1$.
\end{Proposition}

\begin{proof}
Suppose that $m\geq2$. By the claim in the proof of
Proposition~\ref{pro:primepolycube}, ${\mathcal P}$ contains a missing-corner configuration. Since every nonempty $y$-column of ${\mathcal P}$ contains
exactly the unit-cubes $C(i,j,k)$ with $0\leq j\leq m-1$, the two
directions involved in this configuration must be the $x$- and
$z$-directions. Thus, after translating and, if necessary, reversing
these directions, we may assume that 
\[
C(0,0,0),\quad C(1,0,0),\quad C(1,0,1)\in {\mathcal P},
\qquad
C(0,0,1)\notin {\mathcal P}.
\]

Observe that translations and the interchange of the $x$- and
$z$-coordinates preserve Setup~4.1 and induce isomorphisms of the
corresponding posets. Moreover, reversing both the $x$- and
$z$-directions preserves Setup~4.1 and replaces
$Q_{\mathcal P}$ by its dual. Indeed, using the uniform height of
the $y$-columns, the induced correspondence on vertices may also
reverse the $y$-coordinate, giving an order-reversing bijection.
Thus, in all cases, the order complex $\Delta(Q_{\mathcal P})$
is unchanged up to isomorphism.
Set
\[
\alpha=(0,0,1),\qquad a=(0,m,1),\qquad
b=(1,0,2),\qquad \beta=(1,m,2).
\]
Then $a<_{Q_{\mathcal P}}b$: indeed, $\pi(a)<\pi(b)$, while $a$ and $b$
cannot belong to $L_D$ for any cuboid $D\subseteq {\mathcal P}$, since the
only such cuboid would be
$[a\wedge b,a\vee b]=[\alpha,\beta]$, which contains the missing
unit-cube $C(0,0,1)$.

The interval $[\alpha,\beta]_{Q_{\mathcal P}}$ contains the saturated chain
\[
\alpha<(0,1,1)<\cdots<(0,m,1)
<(1,m,1)<\beta,
\]
of length $m+2$. On the other hand, let
\[
a=d_0<d_1<\cdots<d_r=b
\]
be a saturated chain from $a$ to $b$ in $Q_{\mathcal P}$. Since $a<b$, we have
$r\geq1$. Then
\[
\alpha<(0,1,1)<\cdots<(0,m,1)=d_0
<d_1<\cdots<d_r=(1,0,2)
<(1,1,2)<\cdots<\beta
\]
is a saturated chain of length $2m+r$.

Since $m\geq2$ and $r\geq1$, we have $2m+r>m+2$. Hence
$Q_{\mathcal P}$ is not pure, and therefore $\Delta(Q_{\mathcal P})$ is not pure;
hence, it is not Cohen--Macaulay. By Lemma~\ref{lem:initial-order-complex}, this
contradicts the assumption. Thus $m=1$.
\end{proof}

In order to study the Cohen--Macaulay property of $\Delta(Q_\MP)$, we take a closer look at the structure of $Q_\MP$. We begin by observing that $Q_\MP$ is a bounded poset and it is a lattice only in the trivial case.

\begin{Remark}
Let $\MP$ satisfy \Cref{setup:polycube}. It follows from the definition of $Q_\MP$ that the minimal element of sublattice $V(\MP)$ is the minimal element of $Q_\MP$ and the maximal element of $V(\MP)$ is the maximal element of $Q_\MP$. Hence, $Q_\MP$ is bounded.

The poset $Q_\MP$ is a lattice if and only if $\MP$ is a cuboid.
Indeed, if $\MP$ is a cuboid, then $V(\MP)$ is a distributive lattice and every incomparable pair $a,b\in V(\MP)$ belongs to
$L_D$ for some cuboid $D\subseteq \MP$. Hence $Q_\MP = V(\MP)$.

Suppose now that $\MP$ is not a cuboid. Using the notation from the proof of \cref{prop:height-one}, note that both $(1,m,1)$ and $(1,0,2)$ are upper bounds of $\{(0,m,1), (1,0,1)\}$ in $Q_\MP$. If $c = (0,m,1)\vee (1,0,1)$ existed in $Q_\MP$, then we would have $(1,0,1) \leq_{Q_\MP} c \leq_{Q_\MP} (1,m,1)$. Since $\pi((1,0,1)) \leq \pi(c) \leq \pi((1,m,1))$, we get that $c= (1,j,1)$ for some $0\leq j \leq m$. 

If $j\geq 1$, then $c$ and $(1,0,2)$ are incomparable in $\NN^3$, and belong to $L_D$ where  $D = [(1,0,1),(2,j,2)]$ is a cuboid in $\MP$. Thus, $c = (1,0,1)$. But $c$ and $(0,m,1)$ are incomparable in $Q_\MP$ because they are incomparable in $\NN^3$ and belong to $L_{D'}$, where $D'= [(0,0,0),(1,m,1)]$ is a cuboid in $\MP$. Hence $(0,m,1)\vee (1,0,1)$ does not exist, and
$Q_\MP$ is not a lattice.
\end{Remark}

\begin{Notation}\label{not:SP}
Let ${\mathcal P}$ satisfy Setup~\ref{setup:polycube}, and 
set
\[
{\mathcal S}_{\mathcal P}=
\bigl\{\{a,b\}:a\lessdot_{Q_{\mathcal P}}b
\text{ and }a,b\text{ are incomparable in }\mathbb N^3\bigr\},
\]
and
\[
{\mathcal L}_{\mathcal P}=\{a\in V({\mathcal P}):a\in A\text{ for some }A\in\mathcal S_{\mathcal P}\}.
\]
\end{Notation}


Let $\MP$ be a non-cuboid polycube as in Setup~\ref{setup:polycube} such that $K[\MP]$ is Cohen--Macaulay. Then by Proposition~\ref{prop:height-one}, we may assume that ${\mathcal P}\subseteq[(0,0,0),(l,1,n)]$ for some $l,n\geq1$.
Let $\{a,b\}\in\mathcal S_{\mathcal P}$, with $a=(i_1,j_1,k_1)$ and $ b=(i_2,j_2,k_2),$ where $a\lessdot_{Q_{\mathcal P}}b$. Since $a$ and $b$ are incomparable in
$\mathbb N^3$, while $\pi(a)\leq\pi(b)$, we necessarily have $j_1=1, j_2=0.$ Thus every element of $\mathcal S_{\mathcal P}$ is of the form
\[
\{(i_1,1,k_1),(i_2,0,k_2)\}.
\]


The next few lemmas further restrict the possible form of a cover relation
$a\lessdot_{Q_{\mathcal P}} b$ with $\{a,b\}\in S_{\mathcal P}$.

\begin{Lemma}\label{lem:SP-special}
Let $\mathcal P\subseteq[(0,0,0),(l,1,n)]$ satisfy
Setup~\ref{setup:polycube}, and let $\{a,b\}\in\mathcal S_{\mathcal P}$.

\begin{enumerate}
\item If $a=(i,1,k), b=(i,0,k')$ with $ k'>k$ (see Figure~\ref{fig:vertical}), then
\[
\{a,(i+1,0,k+1)\}\in\mathcal S_{\mathcal P}
\qquad\text{and}\qquad
k'\geq k+3.
\]

\item If $a=(i,1,k),b=(i',0,k)$ with $i'>i$ (see Figure~\ref{fig:horizontal}), then
\[
\{a,(i+1,0,k+1)\}\in\mathcal S_{\mathcal P}
\qquad\text{and}\qquad
i'\geq i+3.
\]
\end{enumerate}
\end{Lemma}

\begin{proof}
We prove (1); the proof of (2) is analogous, by interchanging the
$x$- and $z$-coordinates.

\begin{figure}[h]
    \begin{subfigure}[t]{7cm}
        \centering
\begin{tikzpicture}[scale=1.5]

\shadecube{0}{0}{0};
\shadecube{0}{1}{0};
\shadecube{1}{1}{0};
\shadecube{1}{2}{0};


\draw[](0,0,0)--(0+1,0,0)--(0+1,0,0-1)--(0,0,0-1)--(0,0,0);
\draw[](0,0+1,0)--(0+1,0+1,0) -- (0+1,0+1,0-1)--(0,0+1,0-1)--(0,0+1,0);
\draw[](0,0,0)--(0,0+1,0);
\draw[] (0+1,0,0)--(0+1,0+1,0);
\draw[] (0,0,0-1)-- (0,0+1,0-1);
\draw[] (0+1,0,0-1)-- (0+1,0+1,0-1);

\draw[dotted, thick](0,1+1,0)--(0+1,1+1,0)  (0+1,1+1,0-1)--(0,1+1,0-1)--(0,1+1,0);
\draw[dotted, thick](0,1,0)--(0,1+1,0);
\draw[dotted, thick] (0+1,1,0)--(0+1,1+1,0);
\draw[dotted, thick] (0+1,1,0-1)-- (0+1,1+1,0-1);

\draw[dotted, thick](1,1,0)--(1+1,1,0)--(1+1,1,0-1)--(1,1,0-1);
 \draw[dotted, thick] (1+1,1,0-1)-- (1+1,1+1,0-1);

\draw[](1,2,0)--(1+1,2,0)--(1+1,2,0-1)--(1,2,0-1)--(1,2,0);
\draw[](1+1,2+1,0)--(1+1,2+1,0-1)--(1,2+1,0-1)--(1,2+1,0);
\draw[](1,2,0)--(1,2+1,0);
 \draw[] (1+1,2,0-1)-- (1+1,2+1,0-1);
\draw[] (1,2,0-1)--(1,2+1,0-1);

\draw[red] (1,0,-1)--(1,3,0);
\draw[red] (1,0,-1)--(2,1,0);

\draw[thick, blue] (1,3,0)--(2,3,0);
\draw[thick, blue] (2,2,0)--(2,3,0);

\draw[dashed, thick, blue] (2,1,0)--(2,2,0);

\node [draw, circle, fill=black, inner sep=.5pt, label=right:{\tiny{{$(i,0,k)$}}}] at (1,0,0) {};
\node [draw, circle, fill=black, inner sep=.5pt, label=left:{\tiny{{$(i,1,k')$}}}] at (1,3,-1) {}; 

\node [draw, circle, fill=black, inner sep=.5pt, label=right:{\tiny{{$a = (i,1,k)$}}}] at (1,0,-1) {}; 
\node [draw, circle, fill=black, inner sep=.5pt, label=left:{\tiny{{$b = (i,0,k')$}}}] at (1,3,0) {}; 

\node [draw, circle, fill=black, inner sep=.5pt, label=right:{\tiny{{$(i+1,0,k+1)$}}}] at (2,1,0) {}; 
\node [draw, circle, fill=black, inner sep=.5pt, label=right:{\tiny{{$(i+1,0,k+2)$}}}] at (2,2,0) {};
\node [draw, circle, fill=black, inner sep=.5pt, label=right:{\tiny{{$c=(i+1,0,k')$}}}] at (2,3,0) {};

\end{tikzpicture}
	\caption{When $i=i'$}
	\label{fig:vertical}
    \end{subfigure}
            \quad
        \begin{subfigure}[t]{7cm}
            \centering
        \begin{tikzpicture}[scale=1.5]

\shadecube{0}{0}{0};
\shadecube{1}{0}{0};
\shadecube{1}{1}{0};
\shadecube{2}{1}{0};
        
\draw[](0,0,0)--(0+1,0,0)--(0+1,0,0-1)--(0,0,0-1)--(0,0,0);
\draw[](0,0+1,0)--(0+1,0+1,0) -- (0+1,0+1,0-1)--(0,0+1,0-1)--(0,0+1,0);
\draw[] (0,0,0-1)--(0,0+1,0-1);
\draw[](0,0,0)--(0,0+1,0);
\draw[] (0+1,0,0)--(0+1,0+1,0);
\draw[] (0+1,0,0-1)-- (0+1,0+1,0-1);

\draw[dotted, thick](1,0,0)--(1+1,0,0)--(1+1,0,0-1)--(1,0,0-1);
\draw[dotted, thick](1,0+1,0)--(1+1,0+1,0) (1+1,0+1,0-1)--(1,0+1,0-1);
\draw[dotted, thick] (1+1,0,0)--(1+1,0+1,0);
\draw[dotted, thick] (1+1,0,0-1)-- (1+1,0+1,0-1);

\draw[dotted, thick](1+1,1+1,0-1)--(1,1+1,0-1)--(1,1+1,0);
\draw[dotted, thick](1,1,0)--(1,1+1,0);
\draw[dotted, thick] (1,1,0-1)--(1,1+1,0-1);

\draw[](2,1,0)--(2+1,1,0)--(2+1,1,0-1)--(2,1,0-1)--(2,1,0);
\draw[](2+1,1+1,0) -- (2+1,1+1,0-1)--(2,1+1,0-1)--(2,1+1,0);
\draw[](2,1,0)--(2,1+1,0);
\draw[] (2+1,1,0-1)-- (2+1,1+1,0-1);
\draw[] (2,1,0-1)--(2,1+1,0-1);

\draw[red] (0,1,-1)--(1,2,0);
\draw[red] (0,1,-1)--(3,1,0);

\draw[thick, blue] (2,2,0)--(3,2,0);
\draw[thick, blue] (3,1,0)--(3,2,0);

\draw[dashed, thick, blue] (1,1+1,0)--(1+1,1+1,0);

\node [draw, circle, fill=black, inner sep=.5pt, label=left:{\tiny{{$(i,1,k)$}}}] at (0,1,-1) {}; 
\node [draw, circle, fill=black, inner sep=.5pt, label=right:{\tiny{{$(i',0,k)$}}}] at (3,1,0) {}; 

\node [draw, circle, fill=black, inner sep=.5pt, label=left:{\tiny{{$(i+1,0,k+1)$}}}] at (1,2,0) {}; 
\node [draw, circle, fill=black, inner sep=.5pt, label=right:{\tiny{{$(i',0,k+1)$}}}] at (3,2,0) {}; 

\end{tikzpicture}
        \caption{When $k=k'$}\label{fig:horizontal}
    \end{subfigure}
        \caption{}
\end{figure}

Since $\{a,b\}\in\mathcal S_{\mathcal P}$, we have
$a\lessdot_{Q_{\mathcal P}}b$, and $a$ and $b$ do not belong to
$L_D$ for any cuboid $D\subseteq\mathcal P$. Since
$V(\mathcal P)$ is a sublattice of $\NN^3$, $(i,0,k)=a\wedge b$ and
$(i,1,k')=a\vee b$ belong to $V(\mathcal P)$. By convexity,
$(i,0,r),(i,1,r)\in V(\mathcal P)$ for every $k\leq r\leq k'$.
Hence, for each $k\leq r<k'$, at least one of
$C(i-1,0,r)$ and $C(i,0,r)$ belongs to $\mathcal P$.

Set
\[
L_r=C(i-1,0,r)
\qquad\text{and}\qquad
R_r=C(i,0,r).
\]
We first claim that
\begin{equation}\label{claim}
L_k,R_{k'-1}\in\mathcal P,
\qquad
R_k,L_{k'-1}\notin\mathcal P.
\end{equation}
Indeed, $L_k$ and $L_{k'-1}$ cannot both belong to $\mathcal P$,
since convexity would imply that all $L_r$, $k\leq r<k'$, belong
to $\mathcal P$, and then $a$ and $b$ would lie in $L_D$ for the
cuboid $D=[(i-1,0,k),(i,1,k')]$. Similarly, $R_k$ and
$R_{k'-1}$ cannot both belong to $\mathcal P$.

Suppose that $R_k\in\mathcal P$. Then $R_{k'-1}\notin\mathcal P$,
so $L_{k'-1}\in\mathcal P$. Since $V(\mathcal P)$ is a sublattice,
the vertices of $L_k$ are obtained as coordinatewise meets of
vertices of $R_k$ and $L_{k'-1}$. Thus $L_k\in\mathcal P$, which
is impossible by the preceding paragraph. Hence
$R_k\notin\mathcal P$, and therefore $L_k\in\mathcal P$. It follows
in turn that $L_{k'-1}\notin\mathcal P$ and
$R_{k'-1}\in\mathcal P$, proving the claim.

Let $c=(i,1,k+1)$. Since $a<c$ in $\mathbb N^3$ and
$a\lessdot_{Q_{\mathcal P}}b$, we cannot have
$c<_{Q_{\mathcal P}}b$. As $\pi(c)\leq\pi(b)$, Definition~\ref{def:QP}
implies that $c$ and $b$ belong to $L_D$ for some cuboid
$D\subseteq\mathcal P$. Since $L_{k'-1}\notin\mathcal P$, this
cuboid must extend to the right of the plane $x=i$. In particular,
$R_{k+1}\in\mathcal P$.

Set $d=(i+1,0,k+1)$. Since $R_k\notin\mathcal P$, there is no
cuboid $D\subseteq\mathcal P$ such that $a,d\in L_D$. As
$\pi(a)<\pi(d)$, Definition~\ref{def:QP} gives
$a<_{Q_{\mathcal P}}d$. Now we show that this is a cover relation. Suppose that $a<_{Q_{\mathcal P}}u<_{Q_{\mathcal P}}d$ for some $u\in V(\mathcal P)$. Then
$\pi(a)\leq\pi(u)\leq\pi(d)$. If $\pi(u)=\pi(a)$, then either $u=a$, or
$u=(i,0,k)<a$ in $\mathbb N^3$, both contradicting
$a<_{Q_{\mathcal P}}u$. Similarly, $\pi(u)\neq\pi(d)$. Hence $\pi(a)<\pi(u)<\pi(d)$, and it follows that
\[
\pi(u)=(i,k+1)\quad\text{or}\quad \pi(u)=(i+1,k).
\]
Suppose first that $\pi(u)=(i,k+1)$. If $u=(i,0,k+1)$, then
$a$ and $u$ are vertices of $L_k\in\mathcal P$, contradicting
$a<_{Q_{\mathcal P}}u$. If $u=(i,1,k+1)$, then $u$ and $d$
are vertices of $R_{k+1}\in\mathcal P$, contradicting
$u<_{Q_{\mathcal P}}d$. Finally, suppose that $\pi(u)=(i+1,k)$. By
Setup~\ref{setup:polycube}, both $(i+1,0,k)$ and $(i+1,1,k)$
belong to $V(\mathcal P)$. Together with
$L_k,R_{k+1}\in\mathcal P$, this gives all vertices of $R_k$ in
$V(\mathcal P)$, and convexity forces $R_k\in\mathcal P$, a
contradiction. Hence no such $u$ exists. Therefore
$a\lessdot_{Q_{\mathcal P}}d$, and consequently
$\{a,d\}\in\mathcal S_{\mathcal P}$.

It remains to prove that $k'\geq k+3$. If $k'=k+1$, then
$a$ and $b$ both belong to $L_{k}$, a contradiction. Suppose that $k'=k+2$. Using (\ref{claim}) with $k'-1=k+1$ gives
\[
L_k,R_{k+1}\in\mathcal P,
\qquad
R_k,L_{k+1}\notin\mathcal P.
\]
Since $\mathcal P$ is connected, a path of unit-cubes from $L_k$
to $R_{k+1}$ must cross the plane $x=i$. Hence, for some integer
$r$, both $L_r$ and $R_r$ belong to $\mathcal P$. If $r\leq k$,
convexity applied to $R_r$ and $R_{k+1}$ yields
$R_k\in\mathcal P$, a contradiction. If $r\geq k+1$, convexity applied to $L_k$
and $L_r$ yields $L_{k+1}\in\mathcal P$, a contradiction. Therefore $k'\neq k+2$, and hence $k'\geq k+3$.
\end{proof}


\begin{lemma}\label{lem:same-coordinate}
Let $\mathcal P\subseteq[(0,0,0),(l,1,n)]$ satisfy
Setup~\ref{setup:polycube}, and let
$\{a,b\}\in\mathcal S_{\mathcal P}$, where $a=(i,1,k), b=(i',0,k')$. If $i=i'$ or $k=k'$, then
$Q_{\mathcal P}$ is not Cohen--Macaulay and
$\ML_{\mathcal P}$ is not a chain.
\end{lemma}

\begin{proof}
We prove the case $i=i'$; the case $k=k'$ follows analogously by
interchanging the $x$- and $z$-coordinates. Thus
$a=(i,1,k)$ and $b=(i,0,k')$. Since $a\lessdot_{Q_{\mathcal P}}b$, we have $k'>k$. 

Set $d=(i+1,0,k+1)$. By Lemma~\ref{lem:SP-special},
$\{a,d\}\in\mathcal S_{\mathcal P}$ and $k'\geq k+3$. Let
$c=(i+1,0,k')=b\vee d$. Since $V(\mathcal P)$ is a sublattice,
$c\in V(\mathcal P)$. With the notation used in the proof of
Lemma~\ref{lem:SP-special}, we have
$R_{k+1},R_{k'-1}\in\mathcal P$. By convexity,
$R_r\in\mathcal P$ for every $k+1\leq r\leq k'-1$. Hence
\[
d=(i+1,0,k+1)<\cdots<(i+1,0,k'-1)
\]
is a chain in $Q_{\mathcal P}$ contained in the open interval $(a,c)$,
as indicated in blue in Figure~4(a). Moreover, this chain has length at least one because $k'\geq k+3$.

On the other hand, from $R_{k'-1}\in\mathcal P$ it follows that $b\lessdot_{Q_{\mathcal P}}c$.
Since $a\lessdot_{Q_{\mathcal P}}b$ and
$b\lessdot_{Q_{\mathcal P}}c$, the vertex $b$ is isolated in the order complex of $(a,c)$. Hence $\Delta((a,c))$ is a disconnected simplicial complex of dimension $\geq 1$ and therefore is not Cohen--Macaulay.  By Lemma~\ref{lem:intervalCM}, it follows that $Q_{\mathcal P}$ is not Cohen--Macaulay.

Finally, $a,b,d\in \ML_{\mathcal P}$, while $b$ and $d$ are
incomparable in $Q_{\mathcal P}$. Therefore
$\ML_{\mathcal P}$ is not a chain.
\end{proof}

The next lemma shows that if a cover relation in $S_{\mathcal P}$ changes
both the $x$- and $z$-coordinates, then it changes each of them by exactly one.

\begin{lemma}\label{lem:SP-diagonal}
Let $\mathcal P\subseteq[(0,0,0),(l,1,n)]$ satisfy
Setup~\ref{setup:polycube}, and let
$\{a,b\}\in\mathcal S_{\mathcal P}$, where $a=(i_1,1,k_1), b=(i_2,0,k_2)$.
If $i_1\neq i_2$ and
$k_1\neq k_2$, then
$i_2=i_1+1$ and $k_2=k_1+1$.
\end{lemma}

\begin{proof}
By Definition~\ref{def:QP}, we have $i_1<i_2$ and $k_1<k_2$.
Suppose that $i_1+1<i_2$, and set
$p=a\wedge b=(i_1,0,k_1)\in V(\mathcal P)$.

Since $\mathcal P$ is connected and $V(\mathcal P)$ is an induced sublattice of $\mathbb N^3$, we may choose a saturated chain in
$V(\mathcal P)$ from $p$ to $b$ which is also saturated in $\mathbb N^3$. Hence this chain must pass through one of
\[
u=(i_2-1,0,k_2)\qquad\text{or}\qquad
v=(i_2,0,k_2-1).
\]
Therefore at least one of these vertices belongs to
$V(\mathcal P)$.

\begin{figure}
        \centering
\begin{tikzpicture}[scale=2.5]
\unitcube{0}{0}{0};

\node [draw, circle, fill=black, inner sep=.5pt, label=left:{\tiny{{$p=(i_1,0,k_1)$}}}] at (0,0,0) {}; 
\node [draw, circle, fill=black, inner sep=.5pt, label=left:{\tiny{{$a=(i_1,1,k_1)$}}}] at (0,0,-1) {}; 
\node [draw, circle, fill=black, inner sep=.5pt, label=right:{\tiny{{$c=(i_2-1,1,k_1)$}}}] at (1,0,-1) {}; 
\node [draw, circle, fill=black, inner sep=.5pt, label=right:{\tiny{{$u=(i_2-1,0,k_2)$}}}] at (1,1,0) {}; 
\node [draw, circle, fill=black, inner sep=.5pt, label=right:{\tiny{{$(i_2-1,1,k_2)$}}}] at (1,1,-1) {}; 
\node [draw, circle, fill=black, inner sep=.5pt, label=right:{\tiny{{$(i_2-1,0,k_1)$}}}] at (1,0,0) {}; 

\end{tikzpicture}
	\caption{The cuboid $D_1$ in the proof of \cref{lem:SP-diagonal}}
	\label{fig:relations}
\end{figure}
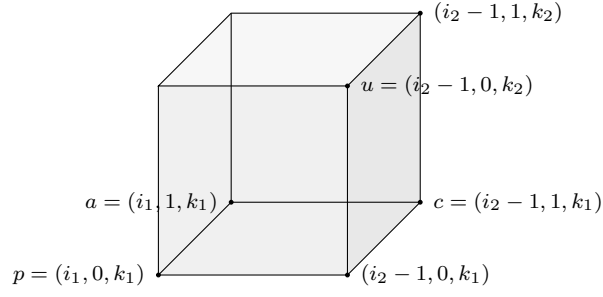

Suppose first that $u\in V(\mathcal P)$. Since $u<b$ and
$a\lessdot_{Q_{\mathcal P}}b$, we cannot have
$a<_{Q_{\mathcal P}}u$. Thus, by Definition~\ref{def:QP},
$a$ and $u$ belong to $L_{D_1}$ for some cuboid
$D_1\subseteq\mathcal P$; see Figure~\ref{fig:relations}.
Necessarily, $D_1=[p,(i_2-1,1,k_2)].$ In particular, $c=(i_2-1,1,k_1)\in V(\mathcal P)$. Since
$a<c$, the cover relation $a\lessdot_{Q_{\mathcal P}}b$ implies
$c\not<_{Q_{\mathcal P}}b$. Hence $c$ and $b$ belong to
$L_{D_2}$ for a cuboid $D_2\subseteq\mathcal P$, necessarily $D_2=[(i_2-1,0,k_1),(i_2,1,k_2)].$ Thus
\[
D_1\cup D_2=[(i_1,0,k_1),(i_2,1,k_2)]\subseteq\mathcal P,
\]
contradicting $\{a,b\}\in\mathcal S_{\mathcal P}$. Therefore
$u\notin V(\mathcal P)$, and consequently $v\in V(\mathcal P)$. 

If $k_2>k_1+1$, the preceding argument, with the $x$- and
$z$-coordinates interchanged and applied to $v$, gives the same
contradiction. Hence $k_2=k_1+1$, so
$v=(i_2,0,k_1)$. By \cref{rem:structure},
$(i_2,1,k_1)\in V(\mathcal P)$. Since also
$a=(i_1,1,k_1)\in V(\mathcal P)$, convexity gives
$c=(i_2-1,1,k_1)\in V(\mathcal P)$. Again $a<c$, so $c\not<_{Q_{\mathcal P}}b$. Therefore $c$ and
$b$ belong to the lattice associated to the cuboid $[(i_2-1,0,k_1),(i_2,1,k_1+1)]\subseteq\mathcal P.$ In particular,
\[
u=(i_2-1,0,k_1+1)=(i_2-1,0,k_2)\in V(\mathcal P),
\]
contradicting $u\notin V(\mathcal P)$. Therefore $i_2=i_1+1$. By symmetry, $k_2=k_1+1$.
\end{proof}


 We are now ready to prove the second necessary condition for
$\Delta(Q_{\mathcal P})$ to be Cohen--Macaulay.

\begin{Theorem}\label{thm:LP-chain}
Let $\mathcal P$ satisfy Setup~\ref{setup:polycube}. If
$Q_{\mathcal P}$ is Cohen--Macaulay, then
${\mathcal L}_{\mathcal P}$ is a chain in $Q_{\mathcal P}$.
\end{Theorem}

\begin{proof}
If $\mathcal P$ is a cuboid, then
$\mathcal S_{\mathcal P}=\emptyset$, and there is nothing to prove.
Thus assume that $\mathcal P$ is not a cuboid. By
Proposition~\ref{prop:height-one}, after a
translation in the $x$- and $z$-coordinates, we may assume that
\[
\mathcal P\subseteq[(0,0,0),(l,1,n)]
\]
for some $l,n\geq1$. By Lemmas~\ref{lem:same-coordinate} and~\ref{lem:SP-diagonal},
every element of $\mathcal S_{\mathcal P}$ is of the form
\[
\{(i,1,k),(i+1,0,k+1)\}.
\]
Set
\[
T_{\mathcal P}
=
\{(i,k):
\{(i,1,k),(i+1,0,k+1)\}\in\mathcal S_{\mathcal P}\}.
\]
Notice that if $(i,k)\in T_{\mathcal P}$, then
$C(i,0,k)\notin\mathcal P$.

Suppose that ${\mathcal L}_{\mathcal P}$ is not a chain.
If $(i,k),(j,h)\in T_{\mathcal P}$ satisfy $i<j$ and $k<h$, then
\[
(i,1,k)<_{Q_{\mathcal P}}(i+1,0,k+1)
\leq_{Q_{\mathcal P}}(j,1,h)
<_{Q_{\mathcal P}}(j+1,0,h+1).
\]
Hence there must exist two points of $T_{\mathcal P}$ which are
not strictly ordered in both coordinates. Choose such a pair
$(i,k),(j,h)$ minimizing $|i-j|+|k-h|$.

After interchanging the two points, we may assume $i\leq j$.
If $i=j$, we may assume $k<h$ (see \cref{fig:casevertical}). If $i<j$, then $k\geq h$; when
$k=h$ (see \cref{fig:casehorizontal}), interchanging the $x$- and $z$-coordinates reduces to the
previous case. Thus it remains to consider
\[
\text{\rm (1)}\quad i=j,\ k<h,
\qquad\text{or}\qquad
\text{\rm (2)}\quad i<j,\ k>h.
\]
By the minimality of the chosen pair, in case~(1) there is no
$(i,s)\in T_{\mathcal P}$ with $k<s<h$. In case~(2), there is no
$(r,s)\in T_{\mathcal P}$ with
\[
i\leq r\leq j,\qquad h\leq s\leq k,
\]
other than $(i,k)$ and $(j,h)$.

\noindent\textit{Case (1).} Suppose that $i=j$ and $k<h$. We claim that
$h\geq k+2$. Assume that $h=k+1$, and set $q=(i,1,k+1), r=(i+1,0,k+1).$ If $q$ and $r$ do not belong to $L_D$ for any cuboid
$D\subseteq\mathcal P$, then
$q\lessdot_{Q_{\mathcal P}}r$, and hence
$\{q,r\}\in\mathcal S_{\mathcal P}$. This contradicts
Lemma~\ref{lem:same-coordinate}, since $q$ and $r$ have the same
third coordinate. Thus $q,r\in L_D$ for some cuboid $D\subseteq\mathcal P$.
Since their third coordinates are equal, $k+1$ is one of the two
endpoints of $D$ in the $z$-direction. Consequently, $D$ contains
either $C(i,0,k)$ or $C(i,0,k+1)$. Both are impossible, since
$(i,k),(i,k+1)\in T_{\mathcal P}$. Hence $h\geq k+2$.

\begin{figure}%
	\begin{subfigure}[t]{6.5cm}
		\centering
		\begin{tikzpicture}[scale=1.5]

\shadecube{0}{0}{0};
\shadecube{0}{1}{0};
\shadecube{1}{1}{0};
\shadecube{2}{1}{0};
\shadecube{2}{2}{0};

\draw[](0,0,0)--(0+1,0,0)--(0+1,0,0-1)--(0,0,0-1)--(0,0,0);
\draw[](0,0+1,0)--(0+1,0+1,0) -- (0+1,0+1,0-1)--(0,0+1,0-1)--(0,0+1,0);
\draw[] (0,0,0-1)--(0,0+1,0-1);
\draw[](0,0,0)--(0,0+1,0);
\draw[] (0+1,0,0)--(0+1,0+1,0);

\draw[dotted, thick](0,1+1,0)--(0+1,1+1,0) -- (0+1,1+1,0-1)--(0,1+1,0-1)--(0,1+1,0);
\draw[dotted, thick](0,1,0)--(0,1+1,0);
\draw[dotted, thick] (0,1,0-1)--(0,1+1,0-1);
\draw[dotted, thick] (0+1,1,0)--(0+1,1+1,0);

 \draw[dotted, thick](1,1,0)--(1+1,1,0)--(1+1,1,0-1)--(1,1,0-1);
 \draw[dotted, thick](1,1+1,0)--(1+1,1+1,0) (1+1,1+1,0-1)--(1,1+1,0-1);

\draw[dotted, thick](2,1,0)--(2+1,1,0)--(2+1,1,0-1)--(2,1,0-1);

\draw[dotted, thick] (2+1,1,0-1)-- (2+1,1+1,0-1);
\draw[dotted, thick] (2,1,0-1)--(2,1+1,0-1);
\draw[dotted, thick] (2+1,1,0)--(2+1,1+1,0);

\draw[](2,2,0)--(2+1,2,0)--(2+1,2,0-1)--(2,2,0-1)--(2,2,0);
\draw[](2+1,2+1,0) -- (2+1,2+1,0-1)--(2,2+1,0-1)--(2,2+1,0);
\draw[](2,2+1,0)--(2+1,2+1,0);
\draw[] (2+1,2,0-1)-- (2+1,2+1,0-1);
\draw[] (2,2,0-1)--(2,2+1,0-1);
\draw[] (2+1,2,0)--(2+1,2+1,0);

\draw[red] (1,0,-1)--(2,1,0);
\draw[red] (1,2,-1)--(2,3,0);

\draw[thick, blue] (0+1,0,0-1)-- (0+1,0+1,0-1);
\draw[dashed, thick, blue] (0+1,1,0-1)-- (0+1,1+1,0-1);

\draw[dashed, thick, blue](2,1,0)--(2,1+1,0);
\draw[thick, blue](2,2,0)--(2,2+1,0);

\node [draw, circle, fill=black, inner sep=.5pt, label=right:{\tiny{{$x = (i,1,k)$}}}] at (1,0,-1) {};
\node [draw, circle, fill=black, inner sep=.5pt, label=below right:{\tiny{{$(i+1,0,k+1)$}}}] at (2,1,0) {};

\node [draw, circle, fill=black, inner sep=.5pt, label=above left:{\tiny{{$(i,1,h)$}}}] at (1,2,-1) {};
\node [draw, circle, fill=black, inner sep=.5pt, label=left:{\tiny{{$y= (i+1,0,h+1)$}}}] at (2,3,0) {};
		\end{tikzpicture}
		\caption{Case (1): $i=j$ and $k<h$;} \label{fig:casevertical}
	\end{subfigure}
	\quad
	\begin{subfigure}[t]{8cm}
		\centering
		\begin{tikzpicture}[scale=1.5]

\shadecube{0}{0}{0};
\shadecube{1}{0}{0};
\shadecube{1}{1}{0};
\shadecube{1}{2}{0};
\shadecube{2}{2}{0};
        
\draw[](0,0,0)--(0+1,0,0)--(0+1,0,0-1)--(0,0,0-1)--(0,0,0);
\draw[](0,0+1,0)--(0+1,0+1,0) -- (0+1,0+1,0-1) (0,0+1,0-1)--(0,0+1,0);
\draw[] (0,0,0-1)--(0,0+1,0-1);
\draw[](0,0,0)--(0,0+1,0);
\draw[] (0+1,0,0)--(0+1,0+1,0);
\draw[] (0+1,0,0-1)-- (0+1,0+1,0-1);

\draw[dotted, thick](1,0,0)--(1+1,0,0)--(1+1,0,0-1)--(1,0,0-1);
\draw[dotted, thick](1,0+1,0)--(1+1,0+1,0) -- (1+1,0+1,0-1);
\draw[dotted, thick] (1+1,0,0)--(1+1,0+1,0);
\draw[dotted, thick] (1+1,0,0-1)-- (1+1,0+1,0-1);

\draw[dotted, thick](1,1,0)--(1,1+1,0);
\draw[dotted, thick] (1+1,1,0)--(1+1,1+1,0);
\draw[dotted, thick] (1+1,1,0-1)-- (1+1,1+1,0-1);
\draw[dotted, thick] (1,1,0-1)--(1,1+1,0-1);

\draw[dotted, thick](1+1,2,0-1)--(1,2,0-1)--(1,2,0);
\draw[dotted, thick](1+1,2+1,0-1)--(1,2+1,0-1)--(1,2+1,0)--(1+1,2+1,0) ;
\draw[dotted, thick](1,2,0)--(1,2+1,0);
\draw[dotted, thick] (1,2,0-1)--(1,2+1,0-1);

\draw[](2+1,2,0)--(2+1,2,0-1)--(2,2,0-1)--(2,2,0);
\draw[](2+1,2+1,0) -- (2+1,2+1,0-1)--(2,2+1,0-1)--(2,2+1,0);
\draw[](2,2,0)--(2,2+1,0);
\draw[](2,2+1,0)--(2+1,2+1,0);
\draw[] (2+1,2,0-1)-- (2+1,2+1,0-1);
\draw[] (2,2,0-1)--(2,2+1,0-1);
\draw[] (2+1,2,0)--(2+1,2+1,0);

\draw[red] (0,1,-1)--(1,2,0);
\draw[red] (2,1,-1)--(3,2,0);

\draw[thick, blue](0+1,0+1,0-1)--(0,0+1,0-1);
\draw[dashed, thick, blue] (1+1,0+1,0-1)--(1,0+1,0-1);

\draw[dashed, thick, blue] (1,2,0)--(1+1,2,0);
\draw[thick, blue] (2,2,0)--(2+1,2,0);

\node [draw, circle, fill=black, inner sep=.5pt, label=left:{\tiny{{$x= (i,1,k)$}}}] at (0,1,-1) {};
\node [draw, circle, fill=black, inner sep=.5pt, label=left:{\tiny{{$(i+1,0,k+1)$}}}] at (1,2,0) {};

\node [draw, circle, fill=black, inner sep=.5pt, label=right:{\tiny{{$(j,1,k)$}}}] at (2,1,-1) {};
\node [draw, circle, fill=black, inner sep=.5pt, label=right:{\tiny{{$y=(j+1,0,k+1)$}}}] at (3,2,0) {};
     \end{tikzpicture}
		\caption{$i < j$ and $k = h$;} \label{fig:casehorizontal}
	\end{subfigure}
		\caption{}
	\label{fig:elementsinQ}
\end{figure}

Set
\[
x=(i,1,k),\qquad y=(i+1,0,h+1).
\]
Since $C(i,0,k)\notin\mathcal P$, the vertices $x$ and $y$ do not
belong to $L_D$ for any cuboid $D\subseteq\mathcal P$. Hence
$x<_{Q_{\mathcal P}}y$. We claim that the open interval $(x,y)$ of $Q_{\mathcal P}$ is the disjoint union of the two chains
\[
(i,1,k+1)<_{Q_{\mathcal P}}\cdots<_{Q_{\mathcal P}}(i,1,h)
\]
and
\[
(i+1,0,k+1)<_{Q_{\mathcal P}}\cdots<_{Q_{\mathcal P}}(i+1,0,h)\,.
\]
Indeed, let $w\in(x,y)$ have second coordinate $1$, and choose a saturated chain from $w$ to $y$. Since $w$ has second coordinate $1$ whereas $y$ has second coordinate $0$, some cover in this chain has lower element with second coordinate $1$ and upper element with second coordinate $0$. Such a cover cannot be a relation in the coordinatewise order on $\mathbb N^3$, and hence belongs to $\mathcal S_{\mathcal P}$. By Lemma~\ref{lem:SP-diagonal}, such a
cover has the form
\[
(r,1,s)\lessdot_{Q_{\mathcal P}}(r+1,0,s+1)
\]
for some $(r,s)\in T_{\mathcal P}$. Since this cover lies in
$[x,y]$, we must have $r=i$ and $k\leq s\leq h$. By the minimality
of the chosen pair, $s=k$ or $s=h$. The case $s=k$ is impossible, because $(i,1,s)\in (x,y)$. Hence $s=h$, and
it follows that $w=(i,1,s)$ for some $k<s\leq h.$ Similarly, if $w$ has second coordinate $0$, a saturated chain
from $x$ to $w$ must contain the cover relation corresponding to $(i,k)\in T_{\mathcal P}$. Thus
$w=(i+1,0,s)$ for some $k+1\leq s\leq h$.

Finally, there is no relation between the two displayed chains.
Indeed, such a relation would give a saturated chain from an
element of the first chain to an element of the second and hence
there is a cover relation from $\MS_\MP$ in $(x,y)$ different from the two chosen
ones.
Thus, $\Delta((x,y))$ is disconnected. Since $h\geq k+2$, it has
dimension at least one, and therefore it is not Cohen--Macaulay.
This contradicts the Cohen--Macaulayness of $Q_{\mathcal P}$ by \cref{lem:intervalCM}.

\noindent\textit{Case (2).}
Suppose that $i<j$ and $k>h$. Write
\[
a=(i,1,k),\quad b=(i+1,0,k+1),\quad
c=(j,1,h),\quad d=(j+1,0,h+1),
\]
and set
\[
x=a\wedge c=(i,1,h),\qquad
y=b\vee d=(j+1,0,k+1),
\]
see Figure~\ref{fig:crossing-case}. Since $V(\mathcal P)$ is a sublattice, $x,y\in V(\mathcal P)$. Moreover, $x<_{Q_{\mathcal P}}y$. Indeed, if
$x,y\in L_D$ for some cuboid $D\subseteq\mathcal P$, then $D=[(i,0,h),(j+1,1,k+1)],$ and hence $C(i,0,k)\in\mathcal P$, contradicting $(i,k)\in T_{\mathcal P}$.

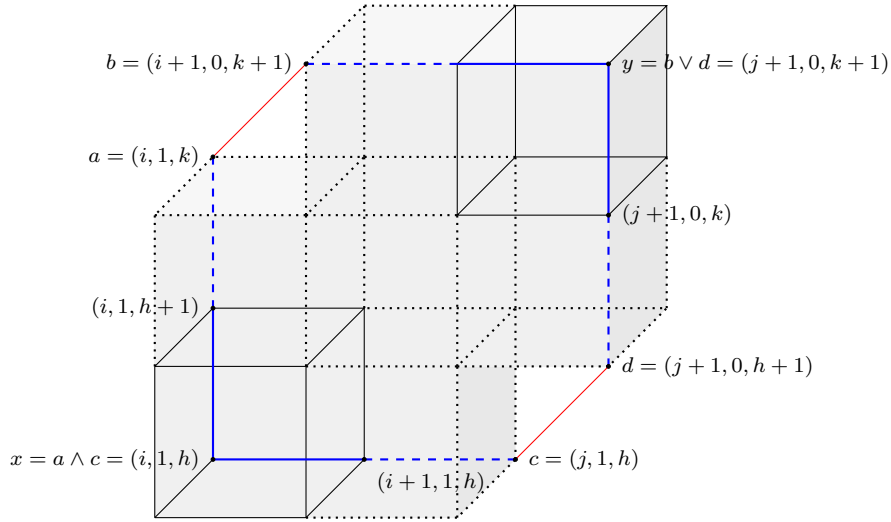
\begin{figure}[h]
\centering
\begin{tikzpicture}[scale=2]

\shadecube{0}{0}{0};
\shadecube{1}{0}{0};
\shadecube{0}{1}{0};
\shadecube{1}{1}{0};
\shadecube{2}{1}{0};
\shadecube{1}{2}{0};
\shadecube{2}{2}{0};

\draw[](0,0,0)--(1,0,0)--(1,0,-1);
\draw[](0,0,-1)--(0,0,0);
\draw[](0,1,0)--(1,1,0)--(1,1,-1)--(0,1,-1)--(0,1,0);
\draw[](0,0,0)--(0,1,0);
\draw[](1,0,0)--(1,1,0);
\draw[](1,0,-1)--(1,1,-1);

\draw[dotted, thick](1,0,0)--(2,0,0)--(2,0,-1);
\draw[dotted, thick](1,1,0)--(2,1,0)--(2,1,-1)--(1,1,-1);
\draw[dotted, thick](2,0,0)--(2,1,0);
\draw[dotted, thick](2,0,-1)--(2,1,-1);

\draw[dotted, thick](0,2,0)--(1,2,0)--(1,2,-1)--(0,2,-1)--(0,2,0);
\draw[dotted, thick](0,1,0)--(0,2,0);
\draw[dotted, thick](1,1,0)--(1,2,0);
\draw[dotted, thick](1,1,-1)--(1,2,-1);

\draw[dotted, thick](1,2,0)--(2,2,0);
\draw[dotted, thick](2,2,-1)--(1,2,-1);
\draw[dotted, thick](2,3,-1)--(1,3,-1)--(1,3,0);
\draw[dotted, thick](1,2,0)--(1,3,0);
\draw[dotted, thick](1,2,-1)--(1,3,-1);

\draw[dotted, thick](2,1,0)--(3,1,0)--(3,1,-1)--(2,1,-1);
\draw[dotted, thick](2,1,0)--(2,2,0);
\draw[dotted, thick](3,1,-1)--(3,2,-1);
\draw[dotted, thick](2,1,-1)--(2,2,-1);

\draw[](2,2,0)--(3,2,0)--(3,2,-1)--(2,2,-1)--(2,2,0);
\draw[](3,3,0)--(3,3,-1)--(2,3,-1)--(2,3,0);
\draw[](2,2,0)--(2,3,0);
\draw[](3,2,-1)--(3,3,-1);
\draw[](2,2,-1)--(2,3,-1);

\draw[red] (0,2,-1)--(1,3,0);
\draw[red] (2,0,-1)--(3,1,0);

\draw[blue, thick] (0,0,-1)--(0,1,-1);
\draw[dashed, thick, blue] (0,1,-1)--(0,2,-1);

\draw[dashed, thick, blue] (1,3,0)--(2,3,0);
\draw[blue, thick] (2,3,0)--(3,3,0);

\draw[blue, thick] (0,0,-1)--(1,0,-1);
\draw[dashed, thick, blue] (1,0,-1)--(2,0,-1);

\draw[dashed, thick, blue] (3,1,0)--(3,2,0);
\draw[blue, thick] (3,2,0)--(3,3,0);

\node [draw, circle, fill=black, inner sep=.5pt,
label=left:{\tiny{$x=a\wedge c=(i,1,h)$}}] at (0,0,-1) {};

\node [draw, circle, fill=black, inner sep=.5pt,
label=left:{\tiny{$(i,1,h+1)$}}] at (0,1,-1) {};

\node [draw, circle, fill=black, inner sep=.5pt,
label=left:{\tiny{$a=(i,1,k)$}}] at (0,2,-1) {};

\node [draw, circle, fill=black, inner sep=.5pt,
label=left:{\tiny{$b=(i+1,0,k+1)$}}] at (1,3,0) {};

\node [draw, circle, fill=black, inner sep=.5pt,
label=below right:{\tiny{$(i+1,1,h)$}}] at (1,0,-1) {};

\node [draw, circle, fill=black, inner sep=.5pt,
label=right:{\tiny{$c=(j,1,h)$}}] at (2,0,-1) {};

\node [draw, circle, fill=black, inner sep=.5pt,
label=right:{\tiny{$d=(j+1,0,h+1)$}}] at (3,1,0) {};

\node [draw, circle, fill=black, inner sep=.5pt,
label=right:{\tiny{$(j+1,0,k)$}}] at (3,2,0) {};

\node [draw, circle, fill=black, inner sep=.5pt,
label=right:{\tiny{$y=b\vee d=(j+1,0,k+1)$}}] at (3,3,0) {};

\end{tikzpicture}
\caption{Case~(2): the crossing case $i<j$ and $k>h$.}
\label{fig:crossing-case}
\end{figure}

We claim that the open interval $(x,y)$ of $Q_{\mathcal P}$ is
the disjoint union of the two chains
\[
(i,1,h+1)<\cdots<a\lessdot b<\cdots<(j,0,k+1)
\]
and
\[
(i+1,1,h)<\cdots<c\lessdot d<\cdots<(j+1,0,k).
\]
The displayed vertices belong to $V(\mathcal P)$ by convexity. Let $w\in(x,y)$ have second coordinate $1$, and choose a saturated
chain from $w$ to $y$. As in Case~(1), this chain contains a cover
\[
(r,1,s)\lessdot_{Q_{\mathcal P}}(r+1,0,s+1)
\]
with $(r,s)\in T_{\mathcal P}$. Since this cover lies in $[x,y]$,
we have $i\le r\le j,$ and $h\le s\le k.$ By the minimality of the chosen pair,
$(r,s)$ is either $(i,k)$ or $(j,h)$. In the first case $w\le a$,
and hence $w=(i,1,s)$ for some $h<s\le k.$ In the second case $w\le c$, and hence $w=(r,1,h)$ for some $i<r\le j.$
Similarly, if $w$ has second coordinate $0$, a saturated chain
from $x$ to $w$ contains one of the same two cover relations in $\MS_\MP$.
If it contains $a\lessdot b$, then
\[
w=(r,0,k+1)\qquad\text{for some }i+1\le r\le j,
\]
whereas if it contains $c\lessdot d$, then
\[
w=(j+1,0,s)\qquad\text{for some }h+1\le s\le k.
\]
Thus every element of $(x,y)$ belongs to one of the two displayed
chains.

There is no relation between the two chains. Elements with the same second coordinate have incomparable $xz-$projections. Any relation from an element with second coordinate $1$ in one chain to an element with second coordinate $0$ in the other would yield, along a saturated chain, an additional element of
$\mathcal S_{\mathcal P}$ in $[x,y]$, contradicting the minimality of the chosen pair. Hence $\Delta((x,y))$ is disconnected. Since each of the two chains contains a cover relation of $\MS_\MP$, $\dim\Delta((x,y))\ge1$. Therefore $\Delta((x,y))$ is not Cohen--Macaulay, a contradiction.
\end{proof}


\subsection{Vertex decomposability}

We now prove the converse of Theorem~\ref{thm:LP-chain}, in the
stronger form that the chain condition on ${\mathcal L}_{\mathcal P}$
implies vertex decomposability.

We begin by recalling an auxiliary fact that will be used repeatedly.

\begin{Lemma}\label{lem:join}
Let $P$ and $Q$ be disjoint finite posets. Then the order complex of the ordinal sum
$P\dirsum Q$ is vertex decomposable if and only if both
$\Delta(P)$ and $\Delta(Q)$ are vertex decomposable.
\end{Lemma}

\begin{proof}
Since
\[
\Delta(P\dirsum Q)=\Delta(P)*\Delta(Q),
\]
where $*$ denotes the join of simplicial complexes, the result follows from
the fact that the join of two simplicial complexes is vertex decomposable if
and only if each factor is vertex decomposable.
\end{proof}


\begin{figure}[h]
    \begin{tikzpicture}[scale=2]
\unitcubebig{0}{0}{0};
\draw[](1,0+1,0)--(1+.75,0+1,0) -- (1+.75,0+1,0-1)--(1,0+1,0-1);
\draw[](3.25,0+1,0)--(1+1.5,0+1,0) -- (1+1.5,0+1,0-1)--(3.25,0+1,0-1) --(3.25,0+1,0);
\draw[dotted, thick](1+1.5,0+1,0)--(1+.75,0+1,0) ;
\draw[dotted, thick] (1+.75,0+1,0-1)--(1+1.5,0+1,0-1);

\node [draw, circle, fill=black, inner sep=.5pt, label=left:{\tiny{{$(0,0,0)$}}}] at (0,0,0) {}; 
\node [draw, circle, fill=black, inner sep=.5pt, label=below:{\tiny{{$(l_1-1,0,0)$}}}] at (1,0,0) {};

 \node [draw, circle, fill=black, inner sep=.5pt, label=left:{\tiny{{$(0,1,0)$}}}] at (0,0,-1) {}; 
 \node [draw, circle, fill=black, inner sep=.5pt, label=left:{\tiny{{$(0,0,t)$}}}] at (0,1,0) {}; 
 \node [draw, circle, fill=black, inner sep=.5pt, label=left:{\tiny{{$(0,1,t)$}}}] at (0,1,-1) {}; 

 \node [draw, circle, fill=black, inner sep=.5pt, label=right:{\tiny{{$(l_1-1,1,0)$}}}] at (1,0,-1) {}; 
 \node [draw, circle, fill=black, inner sep=.5pt, label=above left:{\tiny{{$(l_1-1,1,t)$}}}] at (1,1,-1) {};

 \node [draw, circle, fill=black, inner sep=.5pt, label=above:{\tiny{{$(l_1,1,t)$}}}] at (1.75,1,-1) {}; 
 \node [draw, circle, fill=black, inner sep=.5pt, label=below:{\tiny{{$(r-1,0,t)$}}}] at (2.5,1,0) {}; 
  \node [draw, circle, fill=black, inner sep=.5pt, label=right:{\tiny{{$(r,1,t)$}}}] at (3.25,1,-1) {}; 
 \node [draw, circle, fill=black, inner sep=.5pt, label=right:{\tiny{{$(r,0,t)$}}}] at (3.25,1,0) {}; 

\end{tikzpicture}
	\caption{}
	\label{fig:posetexample}
\end{figure}

\begin{lemma} \label{lem:auxiliary}
 Let $r\geq l_1 \geq 1$ and $t\ge0$. Let $R$ denote the induced subposet of
$\NN^3$ (see \Cref{fig:posetexample}) on the vertex set
\[
[(0,0,0),(l_1-1,1,t)]
\;\cup\;
\{(i,j,t):l_1 \leq i\le r,\ j=0,1\}.
\]
Then $R$ is a distributive lattice. Consequently,
$\Delta(R)$ is vertex decomposable.  
Moreover, $\Delta([(0,0,0), v)_R)$ is vertex decomposable for all $v\in R$, $v\neq (0,0,0)$ .
\end{lemma}

\begin{proof}
    The vertex set of $R$ is closed under coordinatewise minimum and maximum. Thus $R$ is a finite sublattice of
$\mathbb N^3$, hence a distributive lattice. Consequently, $\Delta(R)$ is vertex decomposable by \cite{PB80vertexdecomposable}. The moreover part follows as links of faces of vertex decomposable complexes are vertex decomposable.
\end{proof}

	\begin{figure}[h]
		\begin{subfigure}[t]{11cm}
			\centering	
			    \begin{tikzpicture}[scale=2]
\unitcubebig{1}{0}{0}; 
\draw[](2,0,0)--(2+1,0,0) (2+1,0,0-1)--(2,0,0-1);
\draw[](2,0+1,0)--(2+1,0+1,0) (2+1,0+1,0-1)--(2,0+1,0-1);

\shadecube{2}{0}{0};
\shadecube{3}{1}{0};

\unitcubebig{3}{0}{0};

\draw[](3+1,1+1,0) -- (3+1,1+1,0-1) (3,1+1,0-1)--(3,1+1,0);
\draw[](3,1,0)--(3,1+1,0);
\draw[] (3,1,0-1)--(3,1+1,0-1);

\draw[] (3+1,1,0-1)-- (3+1,1+1,0-1);
\draw[] (3+1,1,0)-- (3+1,1+1,0);

\draw[](3,1+1,0)--(3+.3,1+1,0) (3+1,1+1,0)--(3+.66,1+1,0);
\draw[](3,1+1,0-1)--(3+.33,1+1,0-1) (3+1,1+1,0-1)--(3+.66,1+1,0-1);
\draw[dotted, thick](3+.33,1+1,0)--(3+.66,1+1,0);
\draw[dotted, thick](3+.33,1+1,0-1)--(3+.66,1+1,0-1);

\draw[red] (2,1,-1)--(3,2,0);

\node [draw, circle, fill=black, inner sep=.5pt, label=left:{\tiny{{$(0,0,0)$}}}] at (1,0,0) {}; 
\node [draw, circle, fill=black, inner sep=.5pt, label=below:{\tiny{{$(l_1-1,0,0)$}}}] at (2,0,0) {}; 
\node [draw, circle, fill=black, inner sep=.5pt, label=below:{\tiny{{$(l_1,0,0)$}}}] at (3,0,0) {}; 
\node [draw, circle, fill=black, inner sep=.5pt, label=below:{\tiny{{$(l_2,0,0)$}}}] at (4,0,0) {}; 

\node [draw, circle, fill=black, inner sep=.5pt, label=above left:{\tiny{{$(l_1-1,1,n_1)$}}}] at (2,1,-1) {}; 

\node [draw, circle, fill=black, inner sep=.5pt, label=above right:{\tiny{{$(l_1,1,0)$}}}] at (3,0,-1) {}; 
\node [draw, circle, fill=black, inner sep=.5pt, label=right:{\tiny{{$(l_2,1,0)$}}}] at (4,0,-1) {}; 
\node [draw, circle, fill=black, inner sep=.5pt, label=right:{\tiny{{$(l_2,1,n_1)$}}}] at (4,1,-1) {}; 

\node [draw, circle, fill=black, inner sep=.5pt, label=above left:{\tiny{{$(l_1,0,n_1+1)$}}}] at (3,2,0) {}; 
\node [draw, circle, fill=black, inner sep=.5pt, label=right:{\tiny{{$(l_2,0,n_1+1)$}}}] at (4,2,0) {}; 
\node [draw, circle, fill=black, inner sep=.5pt, label=above:{\tiny{{$(l_1,1,n_1+1)$}}}] at (3,2,-1) {}; 
\node [draw, circle, fill=black, inner sep=.5pt, label=right:{\tiny{{$(l_2,1,n_1+1)$}}}] at (4,2,-1) {}; 

\node [draw, circle, fill=black, inner sep=.5pt, label=left:{\tiny{{$(0,1,0)$}}}] at (1,0,-1) {}; 
\node [draw, circle, fill=black, inner sep=.5pt, label=left:{\tiny{{$(0,0,n_1)$}}}] at (1,1,0) {}; 
\node [draw, circle, fill=black, inner sep=.5pt, label=left:{\tiny{{$(0,1,n_1)$}}}] at (1,1,-1) {}; 
\end{tikzpicture}
			\caption{}\label{fig:vertexdecomhorizontal}
		\end{subfigure}
		\quad
		\begin{subfigure}[t]{10cm}
			\centering	
\begin{tikzpicture}[scale=2]
\unitcubebig{0}{1}{0}; 
\draw[](0,2,0)--(0,2+1,0);
\draw[] (0,2,0-1)--(0,2+1,0-1);

\draw[] (0+1,2,0-1)-- (0+1,2+1,0-1);
\draw[] (0+1,2,0)-- (0+1,2+1,0);
\unitcubebig{0}{3}{0};

\shadecube{0}{2}{0};
\shadecube{1}{3}{0};

\draw[](1,3,0)--(1+1,3,0)--(1+1,3,0-1)--(1,3,0-1);
\draw[](1,3+1,0)--(1+1,3+1,0) -- (1+1,3+1,0-1)--(1,3+1,0-1);

\draw[] (1+1,3,0)--(1+1,3+.33,0) (1+1,3.66,0)--(1+1,3+1,0);
\draw[] (1+1,3,0-1)--(1+1,3+.33,0-1) (1+1,3.66,0-1)--(1+1,3+1,0-1);

\draw[dotted, thick] (1+1,3+.33,0)-- (1+1,3.66,0);
\draw[dotted, thick] (1+1,3+.33,0-1)--(1+1,3.66,0-1);

\draw[red] (1,2,-1)--(2,3,0);

\node [draw, circle, fill=black, inner sep=.5pt, label=below:{\tiny{{$(0,0,0)$}}}] at (0,1,0) {}; 
\node [draw, circle, fill=black, inner sep=.5pt, label=below:{\tiny{{$(l_1,0,0)$}}}] at (1,1,0) {}; 
\node [draw, circle, fill=black, inner sep=.5pt, label=right:{\tiny{{$(l_1,1,0)$}}}] at (1,1,-1) {}; 
\node [draw, circle, fill=black, inner sep=.5pt, label=left:{\tiny{{$(0,1,0)$}}}] at (0,1,-1) {}; 

\node [draw, circle, fill=black, inner sep=.5pt, label=left:{\tiny{{$(0,0,n_1)$}}}] at (0,2,0) {}; 
\node [draw, circle, fill=black, inner sep=.5pt, label=left:{\tiny{{$(0,1,n_1)$}}}] at (0,2,-1) {}; 

\node [draw, circle, fill=black, inner sep=.5pt, label=left:{\tiny{{$(0,0,n_1+1)$}}}] at (0,3,0) {}; 
\node [draw, circle, fill=black, inner sep=.5pt, label=left:{\tiny{{$(0,1,n_1+1)$}}}] at (0,3,-1) {}; 
\node [draw, circle, fill=black, inner sep=.5pt, label=left:{\tiny{{$(0,0,n_2)$}}}] at (0,4,0) {}; 
\node [draw, circle, fill=black, inner sep=.5pt, label=left:{\tiny{{$(0,1,n_2)$}}}] at (0,4,-1) {}; 

\node [draw, circle, fill=black, inner sep=.5pt, label= below left:{\tiny{{$(l_1,0,n_2)$}}}] at (1,4,0) {}; 
\node [draw, circle, fill=black, inner sep=.5pt, label=above :{\tiny{{$(l_1,1,n_2)$}}}] at (1,4,-1) {}; 
\node [draw, circle, fill=black, inner sep=.5pt, label= right:{\tiny{{$(l_1+1,0,n_2)$}}}] at (2,4,0) {}; 
\node [draw, circle, fill=black, inner sep=.5pt, label=right :{\tiny{{$(l_1+1,1,n_2)$}}}] at (2,4,-1) {};

\node [draw, circle, fill=black, inner sep=.5pt, label=right:{\tiny{{$(l_1,1,n_1)$}}}] at (1,2,-1) {}; 
\node [draw, circle, fill=black, inner sep=.5pt, label=right:{\tiny{{$(l_1+1,0,n_1+1)$}}}] at (2,3,0) {}; 
\node [draw, circle, fill=black, inner sep=.5pt, label=right:{\tiny{{$(l_1+1,1,n_1+1)$}}}] at (2,3,-1) {};
\end{tikzpicture}
		\caption{}\label{fig:vertexdecomvertical}
		\end{subfigure}
	\end{figure}

We will also need the following gradedness property, which ensures that
the relevant order complexes are pure.

\begin{Lemma}\label{lem:graded}
Let $\MP\subseteq [(0,0,0),(l,1,n)]$ satisfy \Cref{setup:polycube}, and assume that every element of $\MS_\MP$ has the form 
\{(i,1,k), (i+1,0,k+1)\}. Then $Q_\MP$ is graded with
rank function
\[
\rho(i,j,k)=i+j+k .
\]
In particular, $\Delta(Q_\MP)$ and $\Delta(I)$ for every interval $I$ of
$Q_\MP$ are pure.
\end{Lemma}

\begin{proof}
Let $a\lessdot b$ be a cover relation of $Q_\MP$. If $a\le b$ in $\mathbb N^3$,
then $b=a+e_t$ for some $t\in\{1,2,3\}$.
Otherwise $\{a,b\}\in\MS_\MP$, and $a=(i,1,k)$ and $b=(i+1,0,k+1)$.
In both cases $\rho(b)=\rho(a)+1$. Since $Q_\MP$ is bounded, it follows that all maximal chains of $Q_\MP$ have the same length,
i.e.\ $Q_\MP$ is graded by $\rho$. Intervals of graded posets are graded.
\end{proof}

We are now ready for the main result of this subsection.

\begin{Theorem}\label{thm:chain_vd}
Let $\MP$ be a polycube as in \Cref{setup:polycube} with
$\MP\subseteq[(0,0,0),(l,1,n)]$ for some integers $l,n\ge1$.
If $\ML_\MP$ is a chain in $Q_\MP$, then $\Delta(Q_\MP)$ is vertex decomposable.
\end{Theorem}

\begin{proof}
We proceed by induction on $\#\ML_\MP$. If $\ML_\MP=\varnothing$, then $Q_\MP$
is the set $V(\MP)$ with the order induced from $\mathbb N^3$, hence a
distributive lattice, and $\Delta(Q_\MP)$ is vertex decomposable by
\cite{PB80vertexdecomposable}. Assume therefore that $\#\ML_\MP>0$ and that
the assertion holds for every polycube $\MP'$ satisfying \Cref{setup:polycube},
contained in $[(0,0,0),(l',1,n')]$ for some $l',n'\ge1$, with
$\#\ML_{\MP'}<\#\ML_\MP$. Translating $\MP$ if necessary, we may assume that
$(0,0,0)$ and $(l,1,n)$ are the minimal and the maximal element of $Q_\MP$;
we write $\hat 1=(l,1,n)$.

\medskip\noindent
Since $\ML_\MP\neq\varnothing$, also $\MS_\MP\neq\varnothing$, and by
\Cref{lem:same-coordinate,lem:SP-diagonal} every $\{a,b\}\in\MS_\MP$ has the form
$a=(i,1,k)$, $b=(i+1,0,k+1)$. In particular the minimal element of the chain
$\ML_\MP$ has second coordinate $1$; write it as $(l_1-1,1,n_1)$, so that
$(l_1,0,n_1+1)\in\ML_\MP$ and
\[
(l_1-1,1,n_1)\lessdot(l_1,0,n_1+1)
\]
is a cover relation of $Q_\MP$. In a neighbourhood of this cover the polycube
$\MP$ has one of the two local forms shown in
\Cref{fig:vertexdecomhorizontal,fig:vertexdecomvertical}. We treat the form of
\Cref{fig:vertexdecomhorizontal}; the other case is entirely analogous. 
Note that the vertex set of $Q_\MP$ is the disjoint union of the vertex sets of the following three
intervals:
\begin{align*}
Q_1&=[(0,0,0),(l_1-1,1,n_1)]_{Q_\MP}
     =\{v\in Q_\MP : v\leq (l_1-1,1,n_1) \},\\
B  &=[(l_1,0,0),(l_2,1,n_1)]_{Q_\MP}
     =\{v\in Q_\MP : (l_1,0,0) \leq  v \leq (l_2,1,n_1)\},\\
Q_3&=[(l_1,0,n_1+1),\hat1]_{Q_\MP}
     =\{v\in Q_\MP : v \geq (l_1,0,n_1+1)\},
\end{align*}
and we also set
\[
Q_2=[(l_1,0,0),\hat1]_{Q_\MP}
\qand
\MP_2=\{C(v)\in\MP : v\ge(l_1,0,0)\}.
\]
Note that $V(Q_2) = V(B)\sqcup V(Q_3)$, and $Q_3$ may be viewed as a closed interval in $Q_2$, i.e., $Q_3 = [(l_1,0,n_1+1),\hat{1}]_{Q_2}$ is a closed interval of $Q_2$.
Every element of $Q_1$ is below every element of $Q_3$ in $Q_\MP$, since
\[
u\le(l_1-1,1,n_1)\lessdot(l_1,0,n_1+1)\le v \qforall u \in Q_1, v \in Q_3.
\]

We have partitioned $V(Q_{\mathcal P})$ into three disjoint subsets. We first show that the order complexes of $Q_1$ and $Q_3$ are vertex decomposable and then construct an explicit shedding order on the set $V(B)$.

\medskip\noindent
\emph{Step 1: $\Delta(Q_1)$ is vertex decomposable.}
The vertices in $Q_1$ are exactly the points of
$\{0,\dots,l_1-1\}\times\{0,1\}\times\{0,\dots,n_1\}$, and any two of them are
vertices of the cuboid formed by the unit-cubes $C(i,0,k)$, $0\le i<l_1-1$,
$0\le k<n_1$ (or of a single column of such unit-cubes if $l_1=1$). Hence $Q_1$ is a distributive lattice, 
and $\Delta(Q_1)$ is vertex decomposable by \cite{PB80vertexdecomposable}.

\medskip\noindent
\emph{Step 2: $Q_2=Q_{\MP_2}$, and $\Delta(Q_2)$ is vertex decomposable.}
We first compare the underlying sets. 
Clearly, $V(\mathcal P_2)\subseteq Q_2$. Conversely, let
$a=(a_1,a_2,a_3)\in Q_2$. Choose a unit-cube
$C(i,0,k)\in\mathcal P$ having $a$ as a vertex. Since $a_1\geq l_1$,
either $i\geq l_1$, in which case $C(i,0,k)\in\mathcal P_2$, or
$i=l_1-1$ and $a_1=l_1$. In the latter case, the local form of
$\mathcal P$ shown in \Cref{fig:vertexdecomhorizontal} gives $k<n_1$ and
$C(l_1,0,k)\in\mathcal P_2$, which also contains $a$.
Hence $V(\mathcal P_2)=Q_2$.

Next we compare the orders. Let $a=(a_1,a_2,a_3),b=(b_1,b_2,b_3)\in V(\MP_2)$. If $a\le b$ in
$\mathbb N^3$, then $a\le b$ in both $Q_\MP$ and $Q_{\MP_2}$. Otherwise, by
\Cref{def:QP} the relation $a\le b$ holds in $Q_\MP$ (resp.\ in $Q_{\MP_2}$)
if and only if $\pi(a)\le\pi(b)$ and $a,b$ do not lie in a common cuboid of
$\MP$ (resp.\ of $\MP_2$). It therefore suffices to show that $a,b$ lie in a
common cuboid of $\MP$ if and only if they lie in a common cuboid of $\MP_2$.
One implication is trivial, as every cuboid of $\MP_2$ is a cuboid of $\MP$.
Conversely, let $D=[(i_1,0,k_1),(i_2,1,k_2)]\subseteq\MP$ be a cuboid with
$a,b\in L_D$. If $i_2>l_1$, then $D'=[(\max\{i_1,l_1\},0,k_1),(i_2,1,k_2)]$ is a cuboid all
of whose unit-cubes lie in $\MP_2$, and $a,b\in L_{D'}$ because $a_1,b_1\ge l_1$.
If $i_2=l_1$, then $a,b$ lie on the face $x_1=l_1$ of $D$, and $k_2\le n_1$
since the unit-cubes $C(l_1-1,0,0),\ldots, C(l_1-1,0,n_1-1) \in \MP$ and $C(l_1-1,0,n_1) \notin \MP$; hence
$a,b\in L_{D'}$ for the cuboid $D'=[(l_1,0,k_1),(l_1+1,1,k_2)]$, whose unit-cubes
$C(l_1,0,k)$, $k_1\le k<k_2\le n_1$, belong to $\MP_2$.
Hence $Q_{\MP_2}=Q_2$ as posets; in other words, $Q_{\MP_2}$ is the closed
interval $[(l_1,0,0),\hat1]$ of $Q_\MP$. 

Since $Q_{\MP_2}$ is an interval of $Q_\MP$, the cover relations of $Q_{\MP_2}=Q_2$ are exactly the cover relations of $Q_\MP$ between elements of $Q_2$. Hence, we get $\MS_{\MP_2}\subseteq\MS_\MP$ and
$\ML_{\MP_2}\subseteq\ML_\MP\cap Q_2$. Thus $\ML_{\MP_2}$ is a chain, and as
$(l_1-1,1,n_1)\in\ML_\MP\setminus Q_2$ we have $\#\ML_{\MP_2}<\#\ML_\MP$.
Since $\MP_2$ satisfies \Cref{setup:polycube}, by the induction hypothesis, $\Delta(Q_2)=\Delta(Q_{\MP_2})$ is vertex
decomposable.

\medskip\noindent
\emph{Step 3: $\Delta(Q_3)$ is vertex decomposable.}
Let $M$ be a saturated chain of $Q_2$ from $(l_1,0,0)$ to $(l_1,0,n_1+1)$.
An element $x\in Q_2\setminus M$ is comparable to every element of $M$ if and
only if $x>(l_1,0,n_1+1)$. Thus, since $Q_3 = [(l_1,0,n_1+1),\hat{1}]_{Q_2}$, we get
\[
\lk_{\Delta(Q_2)}(M)=\Delta\bigl(Q_3\setminus\{(l_1,0,n_1+1)\}\bigr),
\]
which is vertex decomposable because $\Delta(Q_2)$ is vertex decomposable and links of faces of vertex
decomposable complexes are vertex decomposable~\cite{PB80vertexdecomposable}.
Since $(l_1,0,n_1+1)$ is the minimal element of $Q_3$, we have the ordinal sum
decomposition $Q_3=\{(l_1,0,n_1+1)\}\oplus\bigl(Q_3\setminus\{(l_1,0,n_1+1)\}\bigr)$,
so that
\[
\Delta(Q_3)=\Delta\bigl(\{(l_1,0,n_1+1)\}\bigr)*\Delta\bigl(Q_3\setminus\{(l_1,0,n_1+1)\}\bigr)
\]
is a cone over a vertex decomposable complex, and hence $\Delta(Q_3)$ is vertex
decomposable by \Cref{lem:join}.

\medskip\noindent
\emph{Step 4: a shedding order on $V(B)$.}
Set
\[
p=l_2-l_1+1
\qquad\text{and}\qquad
s=2p(n_1+1)=|V(B)|.
\]
Order the vertices of $B$ as follows. For
$t=0,\ldots,n_1$, and for each fixed $t$, let
$r=l_2,l_2-1,\ldots,l_1$. For each $r$, list $(r,1,t)$ before $(r,0,t).$ Denote the resulting sequence by
$v_1,\ldots,v_s$; that is, 
\[
\begin{aligned}
&v_1=(l_2,1,0),\; v_2=(l_2,0,0),\; v_3=(l_2-1,1,0),\; v_4=(l_2-1,0,0),\;
\ldots,\;\\& v_{2p-1}=(l_1,1,0),\; v_{2p}=(l_1,0,0),\\
&v_{2p+1}=(l_2,1,1),\; v_{2p+2}=(l_2,0,1),\;\ldots,\;
 v_{4p-1}=(l_1,1,1),\; v_{4p}=(l_1,0,1),\\
&\hspace{2cm}\vdots\\
&v_{s-2p+1}=(l_2,1,n_1),\; v_{s-2p+2}=(l_2,0,n_1),\;\ldots,\;
 v_{s-1}=(l_1,1,n_1),\; v_s=(l_1,0,n_1).
\end{aligned}
\]
Set $\Delta_{0,1}=\Delta(Q_\MP)$ and, for
$1\le i\le s$,
\[
\Delta_{i,1}=\del_{\Delta_{i-1,1}}(v_i),
\qquad
\Delta_{i,2}=\lk_{\Delta_{i-1,1}}(v_i),
\qquad
Q^{(i)}=Q_\MP\setminus\{v_1,\ldots,v_{i-1}\}.
\]
Note that $\Delta_{i-1,1}=\Delta(Q^{(i)})$, since deleting vertices from an
order complex yields the order complex of the induced subposet. 

	\begin{figure}[h]
		\begin{subfigure}[t]{5cm}
			\centering
    \begin{tikzpicture}[scale=1.75]
\draw[](1,0+1,0)--(1+1.25,0+1,0) (1+1.25,0+1,0-1)--(1,0+1,0-1)--(1,0+1,0);

\draw[dotted, thick](1+1.75,0+1,0)--(1+1.25,0+1,0) (1+1.75,0+1,-1)--(1+1.25,0+1,-1);

\draw[](1+1.75,0+1,0)--(1+2,0+1,0) -- (1+2,0+1,0-1)--(1+1.75,0+1,0-1);

\draw[](1,0+.5,0)--(1+1.25,0+.5,0) (1+1.25,0+.5,0-1)--(1,0+.5,0-1)--(1,0+.5,0);

\draw[dotted, thick](1+1.75,0+.5,0)--(1+1.25,0+.5,0) (1+1.75,0+.5,-1)--(1+1.25,0+.5,-1);

\draw[](1+1.75,0+.5,0)--(1+2,0+.5,0) -- (1+2,0+.5,0-1)--(1+1.75,0+.5,0-1);

\draw[] (1,0+1,0)--(1,0,0);
\draw[] (1,0+1,0-1)--(1,0,0-1);
\draw[] (3,0+1,0)--(3,0.5,0);
\draw[] (3,0+1,0-1)--(3,0.5,0-1);

\draw[](1+1,0+.5,0)--(1+1,0+.5,0-1);
\draw[](1+1,0+1,0)--(1+1,0+1,0-1);

\draw[](1,0,0)--(1+1,0,0) -- (1+1,0,0-1)--(1,0,0-1)--(1,0,0);
\draw[](1+1,0+1,0)--(1+1,0,0);
\draw[](1+1,0+1,0-1)--(1+1,0,-1);

 \node [draw, circle, fill=black, inner sep=.5pt, label=below left:{\tiny{{$(r-1,0,t)$}}}] at (1,0,0) {}; 
 \node [draw, circle, fill=black, inner sep=.5pt, label=left:{\tiny{{$(r-1,1,t)$}}}] at (1,0,-1) {}; 

 \node [draw, circle, fill=black, inner sep=.5pt, label=below right:{\tiny{{$(r,0,t)$}}}] at (1+1,0,0) {}; 
 \node [draw, circle, fill=black, inner sep=.5pt, label=below right:{\tiny{{$v_i = (r,1,t)$}}}] at (1+1,0,-1) {}; 

 \node [draw, circle, fill=black, inner sep=.5pt, label=below right:{\tiny{{$(r,1,t+1)$}}}] at (1+1,.5,-1) {}; 
 
 \node [draw, circle, fill=black, inner sep=.5pt, label=right:{\tiny{{$(l_2,1,t+1)$}}}] at (1+2,.5,-1) {}; 
 \node [draw, circle, fill=black, inner sep=.5pt, label=above:{\tiny{{$(l_2,1,n_1)$}}}] at (1+2,1,-1) {}; 

\end{tikzpicture}
	\caption{neighbourhood of $v_i = (r,1,t)$}
	\label{fig:sheddingy=1}
		\end{subfigure}
		\quad
		\begin{subfigure}[t]{10.25cm}
			\centering	
    \begin{tikzpicture}[scale=1.75]
\draw[](1,0+1,0)--(1+1.25,0+1,0) (1+1.25,0+1,0-1)--(1,0+1,0-1)--(1,0+1,0);

\draw[dotted, thick](1+1.75,0+1,0)--(1+1.25,0+1,0) (1+1.75,0+1,-1)--(1+1.25,0+1,-1);

\draw[](1+1.75,0+1,0)--(1+2,0+1,0) -- (1+2,0+1,0-1)--(1+1.75,0+1,0-1);

\draw[](1,0+.5,0)--(1+1.25,0+.5,0) (1+1.25,0+.5,0-1)--(1,0+.5,0-1)--(1,0+.5,0);

\draw[dotted, thick](1+1.75,0+.5,0)--(1+1.25,0+.5,0) (1+1.75,0+.5,-1)--(1+1.25,0+.5,-1);

\draw[](1+1.75,0+.5,0)--(1+2,0+.5,0) -- (1+2,0+.5,0-1)--(1+1.75,0+.5,0-1);

\draw[] (1,0+1,0)--(1,0,0);
\draw[] (1,0+1,0-1)--(1,0,0-1);
\draw[] (3,0+1,0)--(3,0.5,0);
\draw[] (3,0+1,0-1)--(3,0.5,0-1);

\draw[](1+1,0+.5,0)--(1+1,0+.5,0-1);
\draw[](1+1,0+1,0)--(1+1,0+1,0-1);

\draw[](1,0,0)--(1+1,0,0) (1,0,0-1)--(1,0,0);
\draw[](1+1,0+1,0)--(1+1,0,0);

 \node [draw, circle, fill=black, inner sep=.5pt, label=below left:{\tiny{{$(r-1,0,t)$}}}] at (1,0,0) {}; 
 \node [draw, circle, fill=black, inner sep=.5pt, label=left:{\tiny{{$(r-1,1,t)$}}}] at (1,0,-1) {}; 

 \node [draw, circle, fill=black, inner sep=.5pt, label=below right:{\tiny{{$(r,0,t)$}}}] at (1+1,0,0) {}; 

 \node [draw, circle, fill=black, inner sep=.5pt, label=below right:{\tiny{{$(r,0,t+1)$}}}] at (1+1,.5,0) {}; 

  \node [draw, circle, fill=black, inner sep=.5pt, label=right:{\tiny{{$(l_2,1,t+1)$}}}] at (1+2,.5,-1) {}; 
 \node [draw, circle, fill=black, inner sep=.5pt, label=above:{\tiny{{$(l_2,1,n_1)$}}}] at (1+2,1,-1) {}; 

\end{tikzpicture}
	\caption{neighbourhood of $v_i = (r,0,t)$}
	\label{fig:sheddingy=0}            
		\end{subfigure}
     \caption{}   
	\end{figure}

\noindent
\textbf{Claim.} For every $1\le i\le s$, the vertex $v_i$ is a shedding vertex
of $\Delta_{i-1,1}$; moreover, $v_i$ is covered in $Q^{(i)}$ by exactly one
element, namely $w_i=v_i+e_3$.

\noindent
\emph{Proof of the Claim.}
We proceed by induction on $i$.
For $i=1$, we have $\Delta_{0,1}=\Delta(Q_{\mathcal P})$, which is pure by \cref{lem:graded}. 
Let $L$ be a maximal chain of $Q_\MP$ that contains $v_1=(l_2,1,0)$. 
The only element covering $v_1$ in $Q_\MP$ is $(l_2,1,1)$; there is no cover relation from $\MS_P$ starting at $v_1$ because $\ML_\MP$  is a chain, and $v_1$ and $(l_1,0,n_1+1)\in \ML_\MP$ are incomparable in $Q_\MP$ (see \cref{fig:vertexdecomhorizontal}). 
The elements covered by $v_1$ in $Q_\MP$ are $(l_2,0,0)$ and $(l_2-1,1,0)$; exactly one of them lies in $L$. 

If $(l_2,0,0) \in L$, then $(L\setminus \{v_1\}) \cup \{(l_2,0,1)\}$ is a maximal chain of $Q_\MP$. If $(l_2-1,1,0) \in L$, then $(L\setminus \{v_1\}) \cup \{(l_2-1,1,1)\}$ is a maximal chain of $Q_\MP$. Thus, $v_1$ is a shedding vertex of $\Delta_{0,1}$.

Now let $i>1$ and assume that
$v_1,\ldots,v_{i-1}$ have already been shown to be shedding vertices.
Since the deletion of a shedding vertex from a pure simplicial complex
is again pure of the same dimension, it follows inductively that
$\Delta_{i-1,1}$ is pure of dimension
$\dim\Delta(Q_{\mathcal P})$.
Hence its facets are precisely the maximal chains of
$Q^{(i)}$ avoiding $v_1,\ldots,v_{i-1}$. Let $L$ be a facet of $\Delta_{i-1,1}$ that contains $v_i$. We
exhibit an element $u\in Q^{(i)}\setminus L$ such that
$(L\setminus\{v_i\})\cup\{u\}$ is again a facet of $\Delta_{i-1,1}$; this
shows that $v_i$ is a shedding vertex.

Suppose first that $v_i=(r,1,t)$ with $l_1\le r\le l_2$ and $0\le t\le n_1$
(see \Cref{fig:sheddingy=1}). The elements covering $v_i$ in $Q_\MP$ are
$(r,1,t+1)$ and, if $r<l_2$, the element $(r+1,1,t)$; there is no
cover relation from $\MS_P$ starting at $v_i$, because $C(r,0,t)\in\MP$ whenever $r<l_2$. As
$(r+1,1,t)$ has already been deleted, $w_i=(r,1,t+1)$ is the only element
covering $v_i$ in $Q^{(i)}$; thus $w_i\in L$. The elements covered by $v_i$ in
$Q^{(i)}$ are $(r,0,t)$ and $(r-1,1,t)$, the element $(r,1,t-1)$ having been
deleted; exactly one of them lies in $L$.

If $(r,0,t)\in L$, take $u=(r,0,t+1)$; then $(r,0,t)\lessdot u\lessdot w_i$.
If $(r-1,1,t)\in L$ and $(r,t)\neq(l_1,n_1)$, take $u=(r-1,1,t+1)\in V(\MP)$;
then $(r-1,1,t)\lessdot u\lessdot w_i$.
If $(r-1,1,t)=(l_1-1,1,n_1)\in L$, take $u=(l_1,0,n_1+1)$; then
$(l_1-1,1,n_1)\lessdot u\lessdot(l_1,1,n_1+1)=w_i$. In each case $u$ is not
deleted, and hence $(L\setminus\{v_i\})\cup\{u\}$ is a maximal chain of $Q^{(i)}$.

Suppose now that $v_i=(r,0,t)$ (see \Cref{fig:sheddingy=0}). Since $(r,1,t)$
and, if $r<l_2$, $(r+1,0,t)$ have been deleted, $w_i=(r,0,t+1)$ is the only
element covering $v_i$ in $Q^{(i)}$; and since $(r,0,t-1)$ has been deleted,
$(r-1,0,t)$ is the only element covered by $v_i$ in $Q^{(i)}$. No element of $\MS_\MP$ ends at $v_i$, it would have been 
$\{(r-1,1,t-1),v_i\}$ with $C(r-1,0,t-1)\notin \MP$, whereas $C(r-1,0,t-1)\in \MP$.   
Thus $(r-1,0,t),w_i\in L$. If $(r,t)\neq(l_1,n_1)$, take $u=(r-1,0,t+1)$; then
$(r-1,0,t)\lessdot u\lessdot w_i$. If $(r,t)=(l_1,n_1)$, i.e.\ $i=s$, take
$u=(l_1-1,1,n_1)$; then $(l_1-1,0,n_1)\lessdot u\lessdot(l_1,0,n_1+1)=w_s$.
Again $u\in Q^{(i)}$, and $(L\setminus\{v_i\})\cup\{u\}$ is a facet of
$\Delta_{i-1,1}$. This proves the Claim.

\medskip\noindent
\emph{Step 5: every link $\Delta_{i,2}$ is vertex decomposable.}
Fix $1\le i\le s$ and write $v_i=(r,j,t)$. Since
$\Delta_{i-1,1}=\Delta(Q^{(i)})$, we have
\[
\Delta_{i,2}=\lk_{\Delta(Q^{(i)})}(v_i)
=\Delta\bigl(Q^{(i)}_{<v_i}\bigr)*\Delta\bigl(Q^{(i)}_{>v_i}\bigr),
\]
and by \Cref{lem:join} it suffices to show that both factors are vertex
decomposable.

For the upper factor we claim that $Q^{(i)}_{>v_i}=[w_i,\hat1]_{Q_2}$. Every
element $a= (a_1,a_2,a_3)\ge w_i$ of $Q_\MP$ has $a_3 \ge t+1$, so it lies in $Q^{(i)}$,
and $a>v_i$. Conversely, let $a\in Q^{(i)}$ with $a>v_i$. Extend $\{v_i,a\}$
to a maximal chain of $Q^{(i)}$; by $\dim(\Delta_{i-1,1}) =\dim(\Delta(Q_\MP))$ it is a
saturated chain of $Q_\MP$, so the element following $v_i$ in it covers $v_i$
in $Q_\MP$ and lies in $Q^{(i)}$. By the Claim this element is $w_i$, whence
$a\ge w_i$. Now $[w_i,\hat1]_{Q_2}$ is a closed interval of $Q_2$, and
\[
\Delta\bigl([w_i,\hat1]_{Q_2}\bigr)
=\{w_i\}*\lk_{\Delta(Q_2)}\bigl(F\cup\{w_i\}\bigr)
\]
for any maximal chain $F$ of $[(l_1,0,0),w_i)_{Q_2}$. By the same argument
as in Step~3, it is vertex decomposable.

For the lower factor, let $R$ be the poset of \Cref{lem:auxiliary} for the first coordinate $r$ and third coordinate $t$ of $v_i$.
All elements of $Q_\MP$ below $v_i$ lie in the cuboid $D = [(0,0,0),(l_2,1,n_1)]$, where the order relations of $Q_\MP$ are precisely those of $V(D)$. Among them the deleted ones are exactly with first coordinate $\geq l_1$ and third coordinate $<t$; the other deleted vertices are not below $v_i$. Therefore, $Q^{(i)}_{<v_i}=[(0,0,0),v_i)_R$ as a poset, and $\Delta\bigl(Q^{(i)}_{<v_i}\bigr)$ is vertex decomposable by \Cref{lem:auxiliary}. Hence $\Delta_{i,2}$ is vertex
decomposable.

\medskip\noindent
\emph{Step 6: conclusion.}
After deleting $v_1,\ldots,v_s$ we are left with
$V(Q_\MP\setminus B)=V(Q_1)\sqcup V(Q_3)$. Since every element of $Q_1$ lies below
every element of $Q_3$, we get
\[
\Delta_{s,1}=\Delta(Q_1\dirsum Q_3)=\Delta(Q_1)*\Delta(Q_3),
\]
which is vertex decomposable by Steps~1 and~3 and \Cref{lem:join}. Since each
$v_i$ is a shedding vertex of $\Delta_{i-1,1}$ with vertex decomposable link
$\Delta_{i,2}$, descending induction on $i$ shows that every $\Delta_{i,1}$,
$s\ge i\ge0$, is vertex decomposable. In particular
$\Delta(Q_\MP)=\Delta_{0,1}$ is vertex decomposable.
\end{proof}

\subsection{Main theorem}
The preceding results yield the following characterization.

\begin{Theorem}\label{thm:main}
Let $\MP$ satisfy \cref{setup:polycube}. The following are equivalent:
\begin{enumerate}
\item\label{thm:main:CMring}
    $K[\MP]$ is Cohen--Macaulay for any field $K$;
\item\label{thm:main:CMposet}
    $\Delta(Q_\MP)$ is Cohen--Macaulay;
\item \label{thm:main:VD}
    $\Delta(Q_\MP)$ is vertex decomposable;
\item \label{thm:main:polycube} 
    $\MP$ is a cuboid, or $m=1$ and $\ML_\MP$ is a chain in $Q_\MP$.
\end{enumerate}
\end{Theorem}

\begin{proof}
    \eqref{thm:main:CMring} $\iff$ \eqref{thm:main:CMposet} follows from \cref{lem:initial-order-complex} and the equivalence recalled at the beginning of the section. 
    
    \eqref{thm:main:CMposet} $\implies$ \eqref{thm:main:polycube} if $\MP$ is not a cuboid, then $m=1$ by \cref{prop:height-one}, and $\ML_\MP$ is a chain by \cref{thm:LP-chain}.

    \eqref{thm:main:polycube} $\implies$ \eqref{thm:main:VD} if $\MP$ is a cuboid, then $V(\MP)$ is a distributive lattice, and hence $\Delta(Q_\MP)$ is vertex decomposable by \cite{PB80vertexdecomposable}; otherwise, $\Delta(Q_\MP)$ is vertex decomposable by \cref{thm:chain_vd}.
    Finally, \eqref{thm:main:VD} $\implies$ \eqref{thm:main:CMposet} follows from \cite{PB80vertexdecomposable}.
\end{proof}

\section{$L$-convex polycubes}\label{sec:L-convex}

In the previous sections, we studied a class of polycubes whose
defining ideals admit quadratic Gr\"obner bases. We now consider
$L$-convex polycubes and show that their coordinate rings can be
realized as the coordinate rings of polycubes belonging to that
class. The construction is obtained by projecting onto the
$xz$-plane, applying the corresponding construction for
$L$-convex polyominoes, and then lifting the resulting planar
object in the $y$-direction.

We continue to assume the second condition of
Setup~\ref{setup:polycube}, concerning the uniform structure of
the $y$-columns but we no longer assume that $\MP$ is convex and that
$V(\mathcal P)$ is a sublattice of $\mathbb N^3$.

\begin{Setup}\label{setup:L-convex}
Fix an integer $m\geq1$, and let $\mathcal P$ be a
polycube such that every nonempty $y$-column of $\mathcal P$ is of the form $\{C(i,j,k):0\leq j\leq m-1\}$ for some $i,k\in\mathbb N$.
\end{Setup}

For a polycube $\mathcal P$ satisfying
Setup~\ref{setup:L-convex}, let $\mathcal P_{xz}$ denote its
projection onto the $xz$-plane. More precisely,
\[
\mathcal P_{xz}
=
\big\{
[(i,k),(i+1,k+1)]: C(i,j,k)\in\mathcal P
\text{ for some }j
\big\}.
\]

We say that $\mathcal P$ is \emph{$L$-convex} if $\mathcal P_{xz}$ is an
$L$-convex polyomino. Recall that a polyomino is called $L$-convex if
any two of its cells can be connected by a path contained in the
polyomino with at most one change of direction.  We draw a \emph{Ferrers diagram} with a common left edge and weakly increasing row lengths from bottom to top. Thus, up to translation, a Ferrers diagram is a polyomino of the form
\[
\big\{[(i,k),(i+1,k+1)]\, : \, 0\leq k <n, \, 0 \leq i < \lambda_k\big\}, \qquad 1\leq \lambda_0\leq \cdots\leq \lambda_{n-1}\,.
\]

For a polyomino $\mathcal R$, we denote by $K[\mathcal R]$ its
polyomino algebra, as defined in Section~3. We refer to~\cite{EHQR21}
for further details on $L$-convex polyominoes and for the construction
of the Ferrers diagram associated with an $L$-convex polyomino.

Let $\mathcal P$ be an $L$-convex polycube. By
\cite[Theorem~3.1]{EHQR21}, there exists a Ferrers diagram
$\mathcal Q_{xz}$, obtained from $\mathcal P_{xz}$ by the
construction described there, such that
\[
K[\mathcal P_{xz}]
\cong
K[\mathcal Q_{xz}]
\]
as $K$-algebras. We lift $\mathcal Q_{xz}$ to a polycube
$\mathcal Q_{\mathcal P}$ by replacing each unit square of
$\mathcal Q_{xz}$ with a column of $m$ unit-cubes in the
$y$-direction. More precisely,
\[
\mathcal Q_{\mathcal P}
=
\big\{
C(i,j,k): [(i,k),(i+1,k+1)] \in\mathcal Q_{xz},
\ 0\leq j\leq m-1
\big\}.
\]

\begin{Lemma}\label{lem:Q-setup}
The polycube $\mathcal Q_{\mathcal P}$ satisfies
Setup~\ref{setup:polycube}.
\end{Lemma}

\begin{proof}
By construction, every nonempty $y$-column of
$\mathcal Q_{\mathcal P}$ consists of exactly the cubes
\[
C(i,j,k),\qquad 0\leq j\leq m-1.
\]
Moreover, since $\mathcal Q_{xz}$ is a Ferrers diagram, it is
convex. Hence its lift $\mathcal Q_{\mathcal P}$ is convex as
well.

It remains to verify that $V(\mathcal Q_{\mathcal P})$ is a
sublattice of $\mathbb N^3$. Since $\mathcal Q_{xz}$ is a
Ferrers diagram, $V(\mathcal Q_{xz})$ is a sublattice of
$\mathbb N^2$. Furthermore,
\[
V(\mathcal Q_{\mathcal P})
=
\{(i,j,k):(i,k)\in V(\mathcal Q_{xz}),\ 0\leq j\leq m\}.
\]
Thus $V(\mathcal Q_{\mathcal P})$ is closed under coordinatewise meet and join, and hence is a sublattice of $\mathbb N^3$.
\end{proof}

Before studying the relation between
the coordinate rings of $\mathcal P$ and
$\mathcal Q_{\mathcal P}$, we illustrate the construction in a simple example.

\begin{Example}\label{ex:L-convex}
Let $\mathcal P$ be the polycube shown in
Figure~\ref{fig:L-convex-construction}(a). Its projection
$\mathcal P_{xz}$ onto the $xz$-plane, shown in
Figure~\ref{fig:L-convex-construction}(b), is an $L$-convex
polyomino. Hence $\mathcal P$ is an $L$-convex polycube. Applying the construction of \cite{EHQR21} to
$\mathcal P_{xz}$ gives the Ferrers diagram $\mathcal Q_{xz}$ shown in Figure~\ref{fig:L-convex-construction}(c). Lifting
$\mathcal Q_{xz}$ in the $y$-direction gives the polycube $\mathcal Q_{\mathcal P}$ shown in
Figure~\ref{fig:L-convex-construction}(d).
\end{Example}

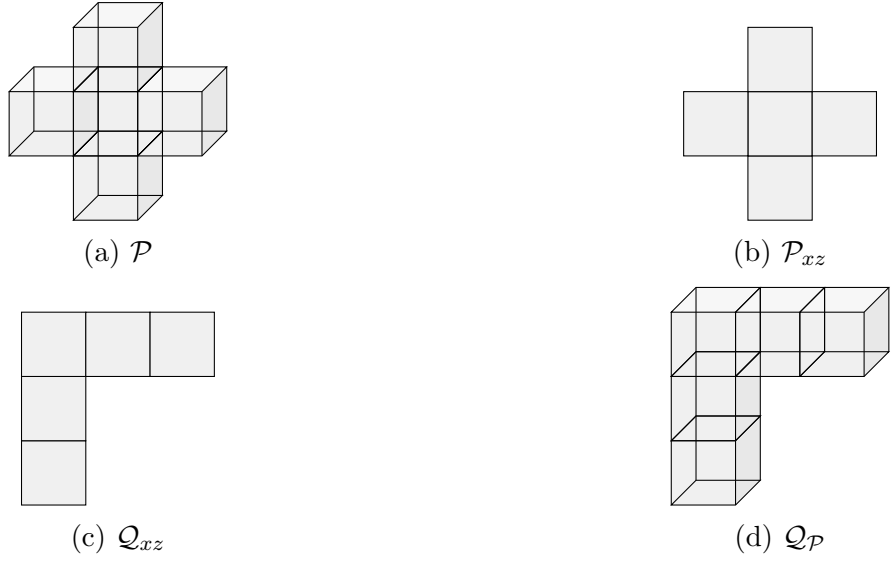
\begin{figure}[htbp]
\centering

\begin{subfigure}[t]{0.45\textwidth}
\centering
\begin{tikzpicture}[scale=.85]
    \unitcube{1}{0}{0};
    \unitcube{1}{1}{0};
    \unitcube{0}{1}{0};
    \unitcube{2}{1}{0};
    \unitcube{1}{2}{0};
\end{tikzpicture}
\caption{$\mathcal P$}
\label{fig:Lconvex-P}
\end{subfigure}
\hfill
\begin{subfigure}[t]{0.45\textwidth}
\centering
\begin{tikzpicture}[scale=.85]
    \foreach \x/\z in {1/0,1/1,1/2,0/1,2/1}
        \draw[fill=gray!12] (\x,\z) rectangle ++(1,1);
\end{tikzpicture}
\caption{$\mathcal P_{xz}$}
\label{fig:Lconvex-Pxz}
\end{subfigure}

\medskip

\begin{subfigure}[t]{0.45\textwidth}
\centering
\begin{tikzpicture}[scale=.85]
    \foreach \x/\z in {0/0,0/1,0/2,1/2,2/2}
        \draw[fill=gray!12] (\x,\z) rectangle ++(1,1);
\end{tikzpicture}
\caption{$\mathcal Q_{xz}$}
\label{fig:Lconvex-Qxz}
\end{subfigure}
\hfill
\begin{subfigure}[t]{0.45\textwidth}
\centering
\begin{tikzpicture}[scale=.85]
    \unitcube{0}{0}{0};
    \unitcube{0}{1}{0};
    \unitcube{0}{2}{0};
    \unitcube{1}{2}{0};
    \unitcube{2}{2}{0};
\end{tikzpicture}
\caption{$\mathcal Q_{\mathcal P}$}
\label{fig:Lconvex-QP}
\end{subfigure}

\caption{The construction of $\mathcal Q_{\mathcal P}$ from the
$L$-convex polycube $\mathcal P$.}
\label{fig:L-convex-construction}
\end{figure}

\begin{Remark}\label{rem:cuboid-rectangle}
Let $\mathcal P$ be a polycube satisfying
Setup~\ref{setup:L-convex}. We use the terminology of \cite{EHQR21}: an \emph{inner interval}
of a polyomino $P$ is an axis-parallel rectangle whose
constituent unit squares all belong to $P$. If $D=[(i_1,j_1,k_1),(i_2,j_2,k_2)]
\subseteq\mathcal P$ is a cuboid, then its projection $[(i_1,k_1), (i_2,k_2)]$ is an inner interval of $\mathcal P_{xz}$. Conversely, if $[(i_1,k_1), (i_2,k_2)]$ is an inner interval of $\mathcal P_{xz}$, then the cuboid $[(i_1,j_1,k_1),(i_2,j_2,k_2)]
\subseteq\mathcal P$ for every $0\leq j_1<j_2\leq m$.
\end{Remark}

We are now ready to compare the coordinate rings of
$\mathcal P$ and $\mathcal Q_{\mathcal P}$.

\begin{Theorem}\label{thm:L-convex-isomorphism}
Let $\mathcal P$ be an $L$-convex polycube satisfying
Setup~\ref{setup:L-convex}. Then
\[
K[\mathcal P]\cong K[\mathcal Q_{\mathcal P}]
\]
as standard graded $K$-algebras.
\end{Theorem}

\begin{proof}
By the construction in the proof of \cite[Theorem~3.1]{EHQR21}, there are bijections $\alpha$ and $\beta$ of the sets of horizontal and vertical coordinates such that 
\[
\sigma:V(\mathcal P_{xz})\longrightarrow V(\mathcal Q_{xz}), \qquad \sigma(i,k) = (\alpha(i),\beta(k))
\]
is a bijection.
Since $\MP_{xz}$ is L-convex and $\mathcal Q_{xz}$ is a Ferrers diagram, both of them are convex polyominoes.
If $R = [(i_1,k_1),(i_2,k_2)]$ is an inner interval in $\MP_{xz}$, then its four corners map onto the four corners of 
\[
R' = \big[\big(\min\{\alpha(i_1),\alpha(i_2)\},\min\{\beta(k_1),\beta(k_2)\}\big),
\big(\max\{\alpha(i_1),\alpha(i_2)\},\max\{\beta(k_1),\beta(k_2)\}\big)\big]\,.
\]
Since all four corner vertices of $R'$ are in $V(\MQ_{xz})$ and $\MQ_{xz}$ is convex, $R'$ is an inner interval of $\MQ_{xz}$.
By applying the same argument to $\sigma^{-1}$, we get $\sigma$ induces a bijection between the inner intervals of $\MP_{xz}$ and $\MQ_{xz}$ through corner vertices. Thus, $\sigma$ gives an isomorphism $K[\mathcal P_{xz}]\cong K[\mathcal Q_{xz}]$~\cite[Theorem~3.1]{EHQR21}.

Since every nonempty $y$-column of both $\mathcal P$ and
$\mathcal Q_{\mathcal P}$ has height $m$,
\[
V(\mathcal P)
=
\{(i,j,k):(i,k)\in V(\mathcal P_{xz}),\ 0\leq j\leq m\},
\]
and
\[
V(\mathcal Q_{\mathcal P})
=
\{(i,j,k):(i,k)\in V(\mathcal Q_{xz}),\ 0\leq j\leq m\}.
\]
Hence $\sigma$ extends naturally to a bijection $\widehat{\sigma}:V(\mathcal P)
\longrightarrow V(\mathcal Q_{\mathcal P})$ by $\widehat{\sigma}(i,j,k)=(\alpha(i),j,\beta(k))$.

By uniform thickness and \cref{rem:cuboid-rectangle}, the induced correspondence of rectangle corners gives a
bijection $D \to D'$ between inner cuboids of $\MP$ and $\MQ_\MP$, satisfying $\widehat{\sigma}(L_D) = L_{D'}$. 
Indeed, let $D=[(i_1,j_1,k_1),(i_2,j_2,k_2)]\subseteq\mathcal P$ be a cuboid, see
\Cref{fig:polycube}. By \Cref{rem:cuboid-rectangle}, its projection $R=[(i_1,k_1),(i_2,k_2)]$ is an inner interval of $\mathcal P_{xz}$ (\Cref{fig:polyomino}), so there is an inner interval $R'=[(i_1',k_1'),(i_2',k_2')]$ of $\mathcal Q_{xz}$ (\Cref{fig:ferrerpolyomino}) such that $\sigma$ maps corner vertices of $R$ onto corner vertices of $R'$. As
$\mathcal Q_{\mathcal P}$ satisfies \Cref{setup:L-convex}, lifting $R'$
between the levels $j_1$ and $j_2$ gives, again by
\Cref{rem:cuboid-rectangle}, a cuboid
\[
D'=[(i_1',j_1,k_1'),(i_2',j_2,k_2')]\subseteq\mathcal Q_{\mathcal P}
\qquad\text{with}\qquad \widehat\sigma(L_D)=L_{D'},
\]
see \Cref{fig:ferrerpolycube}. Since $\sigma^{-1}$ also preserves inner intervals, the same argument applied to $\widehat\sigma^{-1}$ shows that every cuboid of $\mathcal Q_{\mathcal P}$ arises in this way.

Now, let $S_{\mathcal P}=K[x_a:a\in V(\mathcal P)]$ and $S_{\mathcal Q}=K[X_b:b\in V(\mathcal Q_{\mathcal P})].$ The bijection $\widehat{\sigma}$ induces a graded $K$-algebra
isomorphism $
\phi:S_{\mathcal P}\longrightarrow S_{\mathcal Q},$ defined by $\phi(x_a)=X_{\widehat{\sigma}(a)}.$ It is enough to show that $\phi(I_{\mathcal P})=I_{\mathcal Q_{\mathcal P}}.$

For a cuboid $D$ of $\MP$, let $J_D\subseteq S_{\mathcal P}$ be the ideal generated by
the binomials $f_{a,b}=x_ax_b-x_{a\wedge b}x_{a\vee b}$ with $a,b\in L_D$
incomparable. We first claim that $J_D$ is generated by the binomials
\begin{equation}\label{eq:multiset}
x_ax_b-x_cx_d,
\qquad a,b,c,d\in L_D,\quad
\{a_t,b_t\}=\{c_t,d_t\}\ \text{ for } t=1,2,3,
\end{equation}
where $a=(a_1,a_2,a_3)$, etc. Indeed, if $\{a_t,b_t\}=\{c_t,d_t\}$ for all
$t$, then $a\wedge b=c\wedge d$ and $a\vee b=c\vee d$, since these are
computed coordinatewise; hence $x_ax_b-x_cx_d=f_{a,b}-f_{c,d}$, where we set
$f_{a,b}=0$ if $a,b$ are comparable. Conversely each $f_{a,b}$ is of the form
\eqref{eq:multiset}, because $\{a_t,b_t\}=\{(a\wedge b)_t,(a\vee b)_t\}$.

The map $\widehat{\sigma}$ acts independently in each coordinate, so it preserves these coordinate-pair equalities
in both directions. Hence $\phi(J_D) = J_{D'}$, where $J_{D'}$ is defined similarly.
Summing over all inner cuboids yields
\[
\phi(I_{\mathcal P})=\sum_{D}\phi(J_D)=\sum_{D'}J_{D'}=I_{\mathcal Q_{\mathcal P}},
\]
which completes the proof.
\end{proof}

	\begin{figure}[h]
		\begin{subfigure}[t]{7cm}
			\centering
\begin{tikzpicture}[scale=2]
\unitcubeBig{1}{1}{0} ;

\draw[dotted, thick](1-.2,1,0)--(1,1,0);
\draw[dotted, thick](1,1-.2,0)--(1,1,0);
\draw[dotted, thick](1,1,0+.3)--(1,1,0);

\draw[dotted, thick](2+.2,2,-1)--(2,2,-1);
\draw[dotted, thick](2,2+.2,-1)--(2,2,-1);
\draw[dotted, thick](2,2,-1-.3)--(2,2,-1);

  \node [draw, circle, fill=black, inner sep=.5pt, label=above left:{\tiny{{$(i_1,j_1,k_1)$}}}] at (1,1,0) {}; 
  \node [draw, circle, fill=black, inner sep=.5pt, label=above right:{\tiny{{$(i_2,j_2,k_2)$}}}] at (1+1,1+1,-1) {}; 


\end{tikzpicture}
\caption{The cuboid $[(i_1,j_1,k_1), (i_2,j_2,k_2)]$ in $\MP$}\label{fig:polycube}
\end{subfigure}
\quad
\begin{subfigure}[t]{8cm}
\centering	
\begin{tikzpicture}[scale=1]
\fill[gray!12] (1,0) rectangle (4,3);
    \draw[] (1,1) -- (1,0)--(2,0) (3,0)--(4,0)--(4,1) (4,2)--(4,3)--(3,3) (2,3)--(1,3)--(1,2) ;
    \draw[dotted, thick] (2,0) -- (3,0) (4,1) -- (4,2) (3,3) -- (2,3) (1,1) --(1,2);

  \node [draw, circle, fill=black, inner sep=.5pt, label=left:{\tiny{{$(i_1,k_1)$}}}] at (1,0) {}; 
  \node [draw, circle, fill=black, inner sep=.5pt, label=right:{\tiny{{$(i_2,k_2)$}}}] at (4,3) {};    

  \node [draw, circle, fill=black, inner sep=.5pt, label=right:{\tiny{{$(i_2,k_1)$}}}] at (4,0) {}; 
  \node [draw, circle, fill=black, inner sep=.5pt, label=left:{\tiny{{$(i_1,k_2)$}}}] at (1,3) {};  
\end{tikzpicture}
\caption{The inner interval $[(i_1,k_1),(i_2,k_2)]$ in $\MP_{xz}$}\label{fig:polyomino}            
\end{subfigure}
\quad        
	\begin{subfigure}[t]{7cm}
	\centering
\begin{tikzpicture}[scale=1]
\fill[gray!12] (1,0) rectangle (4,3);
    \draw (1,1) -- (1,0)--(2,0) (3,0)--(4,0)--(4,1) (4,2)--(4,3)--(3,3) (2,3)--(1,3)--(1,2) ;
    \draw[dotted, thick] (2,0) -- (3,0) (4,1) -- (4,2) (3,3) -- (2,3) (1,1) --(1,2);

  \node [draw, circle, fill=black, inner sep=.5pt, label=left:{\tiny{{$(i'_1,k'_1)$}}}] at (1,0) {}; 
  \node [draw, circle, fill=black, inner sep=.5pt, label=right:{\tiny{{$(i'_2,k'_2)$}}}] at (4,3) {};    

  \node [draw, circle, fill=black, inner sep=.5pt, label=right:{\tiny{{$(i'_2,k'_1)$}}}] at (4,0) {}; 
  \node [draw, circle, fill=black, inner sep=.5pt, label=left:{\tiny{{$(i'_1,k'_2)$}}}] at (1,3) {};  



\end{tikzpicture}
\caption{ The inner interval $[(i'_1,k'_1),(i'_2,k'_2)]$ in polyomino $\MQ_{xz}$}\label{fig:ferrerpolyomino}
\end{subfigure}
		\quad
\begin{subfigure}[t]{7cm}
			\centering	
\begin{tikzpicture}[scale=2]
\unitcubeBig{1}{1}{0} ;

\draw[dotted, thick](1-.2,1,0)--(1,1,0);
\draw[dotted, thick](1,1-.2,0)--(1,1,0);
\draw[dotted, thick](1,1,0+.3)--(1,1,0);

\draw[dotted, thick](2+.2,2,-1)--(2,2,-1);
\draw[dotted, thick](2,2+.2,-1)--(2,2,-1);
\draw[dotted, thick](2,2,-1-.3)--(2,2,-1);

  \node [draw, circle, fill=black, inner sep=.5pt, label=above left:{\tiny{{$(i'_1,j_1,k'_1)$}}}] at (1,1,0) {}; 
  \node [draw, circle, fill=black, inner sep=.5pt, label=above right:{\tiny{{$(i'_2,j_2,k'_2)$}}}] at (1+1,1+1,-1) {}; 

\end{tikzpicture}
\caption{The cuboid $[(i'_1,j_1,k'_1), (i'_2,j_2,k'_2)]$ in polycube $\MQ_\MP$}\label{fig:ferrerpolycube}            
\end{subfigure}        
    \caption{}   
\end{figure}

Combining Lemma~\ref{lem:Q-setup},
Theorem~\ref{thm:L-convex-isomorphism}, and
Theorem~\ref{thm:GB}, we immediately obtain the following
corollary.

\begin{Corollary}\label{cor:L-convex-GB}
Let $\mathcal P$ be an $L$-convex polycube satisfying
Setup~\ref{setup:L-convex}. Then $I_{\mathcal P}$ admits a
quadratic Gr\"obner basis with respect to a graded reverse
lexicographic order. In particular, $K[\mathcal P]$ is Koszul.
\end{Corollary}

We close the article with a characterization of the Cohen--Macaulay
$L$-convex polycubes. It is obtained by combining
\Cref{thm:L-convex-isomorphism} and \Cref{thm:main} with the structure of Ferrers diagrams.

\begin{Theorem}\label{thm:LconvexCM}
Let $\MP$ be an $L$-convex polycube satisfying \Cref{setup:L-convex}. Then
$K[\MP]$ is Cohen--Macaulay if and only if $\MP$ is a cuboid or $m=1$.
\end{Theorem}

\begin{proof}
By \Cref{lem:Q-setup} and \Cref{thm:L-convex-isomorphism}, $K[\MP]\cong
K[\MQ_\MP]$, where $\MQ_\MP$ satisfies \Cref{setup:polycube} with the same
$m$. In particular, $I_\MP$ is prime if and only if $I_{\MQ_\MP}$ is, so by
\Cref{pro:primepolycube} $\MP$ is a cuboid if and only if $\MQ_\MP$ is. 
By \Cref{thm:main} it therefore suffices to show that $\ML_{\MQ_\MP}$ is a chain when $m=1$.

\begin{figure}[h]
    \centering
    \begin{tikzpicture}[scale=.8]
  \def\rows{2,3,3,5}   
  \pgfmathtruncatemacro\lmax{max(\rows)}
  \pgfmathtruncatemacro\nrows{dim({\rows})}
  \foreach \l [count=\k from 0, remember=\l as \lp (initially 1000)] in \rows {
    \foreach \i in {1,...,\l} \draw[fill=gray!12] (\i-1,\k) rectangle ++(1,1);
    \ifnum\lp<\l\relax 
      \draw[red,thick] (\lp,\k-1) -- (\lp+1,\k);
      \fill[red] (\lp,\k-1) circle (1.6pt) (\lp+1,\k) circle (1.6pt);
    \fi
  }
  \node[below left,font=\small] at (0,0) {$(0,0)$};
 \node[above right,font=\small] at (5,4) {$(5,4)$};
\end{tikzpicture}
    \caption{The Ferrers diagram $\MQ_{xz}$ with $\lambda_0=2,\lambda_1=3,\lambda_2=3,\lambda_3=5$}
    \label{fig:placeholder}
\end{figure}
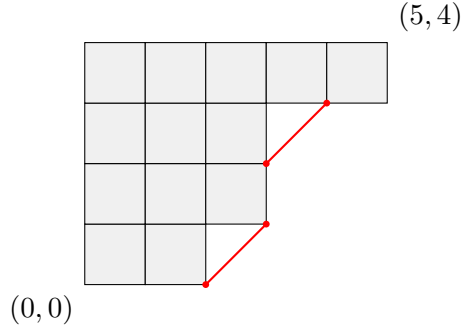

Assume $m=1$. After a translation, assume that 
\[
\MQ_\MP = \{C(i,0,k) : 0 \leq k < n, \, 0\leq i < \lambda_k\} \qwhere 1\leq \lambda_0 \leq \cdots \leq \lambda_{n-1}\,. 
\]
See \cref{fig:placeholder} for a schematic showing of the Ferrers diagram $\MQ_{xz}$. 
Let $\{a,b\}\in\MS_{\MQ_\MP}$ with $a=(i,1,k)$ and $b=(i',0,k')$. Then
$i\leq i'$, $k\leq k'$, and there is no cuboid $D\subseteq \MQ_\MP$ such that $a,b\in L_D$.

If $i = i'$ and $k<k'$, then $a$ and $b$ are vertices of $L_D$ for the cuboid $D = [(i-1,0,k),(i,1,k')] \subseteq \MQ_\MP$ when $i\geq 1$, and $a$ and $b$ are vertices of $L_{D'}$ for the cuboid $D' = [(0,0,k),(1,1,k')] \subseteq \MQ_\MP$ when $i = 0$. 
If $i < i'$ and $k=k'$, then $a$ and $b$ are vertices of $L_D$ for the cuboid $D = [(i,0,k),(i',1,k+1)] \subseteq \MQ_\MP$ when $k<n$, and $a$ and $b$ are vertices of $L_{D'}$ for the cuboid $D' = [(i,0,n-1),(i',1,n)] \subseteq \MQ_\MP$ when $k = n$.

Hence, \cref{lem:SP-diagonal} gives $i'=i+1$ and $k'=k+1$. The unique unit-cube having $a$ and $b$ as corners is $C(i,0,k)$, so it is not in $\MQ_\MP$. 
Since $a= (i,1,k)\in V(\MQ_\MP)$, we have $i\leq \lambda_k$. Moreover, since $C(i,0,k)$ is not in $\MQ_\MP$, we have $i\geq \lambda_k$. Thus $i=\lambda_k$, and hence,
\[
\{a,b\} = \{(\lambda_k,1,k),(\lambda_k+1,0,k+1)\}\,.
\]

Now let $\{(\lambda_k,1,k),(\lambda_k+1,0,k+1)\}$ and
$\{(\lambda_h,1,h),(\lambda_h+1,0,h+1)\}$ be elements of $\MS_{\MQ_\MP}$ with
$k<h$. Since $h\ge k+1$ and $\lambda_h\ge\lambda_{k+1}>\lambda_k$, we have
$(\lambda_k+1,0,k+1)\le(\lambda_h,1,h)$ in $\mathbb N^3$. Hence
\[
(\lambda_k,1,k)<(\lambda_k+1,0,k+1)\le(\lambda_h,1,h)<(\lambda_h+1,0,h+1)
\]
in $Q_{\MQ_\MP}$, and $\ML_{\MQ_\MP}$ is a chain.
\end{proof}


\end{document}